\documentclass[10pt]{article}
\usepackage{amssymb,amsmath,amsthm,youngtab}
\usepackage{verbatim,url}
\usepackage{graphicx,color}
\newtheorem*{definition}{Definition}

\newtheorem{claim}{Claim}
\newtheorem{theorem}{Theorem}
\newtheorem{lemma}[theorem]{Lemma}
\newtheorem*{fact}{Fact}
\newtheorem{conjecture}[theorem]{Conjecture}
\newtheorem{proposition}[theorem]{Proposition}
\newtheorem{corollary}[theorem]{Corollary}
\newcommand{\ud}{\,\mathrm{d}}
\DeclareMathOperator{\Cay}{Cay}
\DeclareMathOperator{\Span}{Span}
\DeclareMathOperator{\sgn}{sgn}
\newcommand{\domleq}{\trianglelefteq}

\newcommand{\domgeq}{\trianglerighteq}
\providecommand{\cA}{\mathcal{A}}
\providecommand{\cC}{\mathcal{C}}
\providecommand{\cE}{\mathcal{E}}
\providecommand{\cM}{\mathcal{M}}
\providecommand{\cX}{\mathcal{X}}
\providecommand{\cY}{\mathcal{Y}}
\providecommand{\boundary}[1]{|\partial #1|}
\providecommand{\edges}[2]{e(#1,#2)}
\newcommand{\round}[1]{\operatorname{round}(#1)}

\newenvironment{myproof}[1][\proofname]{\proof[#1]}{\endproof}

\begin{document}
\title{Low-degree Boolean functions on $S_n$, with an application to isoperimetry}
\author{David Ellis\footnote{School of Mathematical Sciences, Queen Mary, University of London, 327 Mile End Road, London, E1 4NS, UK.}, Yuval Filmus\footnote{Computer Science Department, Technion -- Israel Institute of Technology, Technion City, Haifa 3200003, Israel. This research was supported in part by the Canadian Friends of the Hebrew University / University of Toronto Permanent Endowment.}, and Ehud Friedgut\footnote{Faculty of Mathematics and Computer Science, Weizmann Institute of Science, Rehovot 76100, Israel. This research was supported in part by I.S.F. grant 0398246, and B.S.F. grant 2010247.}}
\date{November 2015}
\maketitle

\begin{abstract}
We prove that Boolean functions on \(S_n\), whose Fourier transform is highly concentrated on irreducible representations indexed by partitions of $n$ whose largest part has size at least \(n-t\), are close to being unions of cosets of stabilizers of \(t\)-tuples. We also obtain an edge-isoperimetric inequality for the transposition graph on \(S_n\) which is asymptotically sharp for subsets of $S_n$ of size $n!/\textrm{poly}(n)$, using eigenvalue techniques. We then combine these two results to obtain a sharp edge-isoperimetric inequality for subsets of $S_n$ of size \((n-t)!\), where \(n\) is large compared to \(t\), confirming a conjecture of Ben Efraim in these cases.
\end{abstract}

\section{Introduction}
This paper (together with \cite{EFF1} and \cite{EFF2}) is part of a trilogy dealing with stability and `quasi-stability' results concerning Boolean functions on the symmetric group, which are (in a sense) of `low complexity'.

Let us begin with some notation and definitions which will enable us to present the Fourier-theoretic context of our results.

If $X$ is a set and $S \subset X$, we denote by $1_{S}$ the {\em characteristic function} (or {\em indicator function}) of $S$, i.e.\ $1_S:X \to \{0,1\}$ with $1_S(x)=1$ if and only if $x \in S$. If $f,g:X \to \{0,1\}$ are Boolean functions, we say that $f$ and $g$ are {\em $\epsilon$-close} if $|\{x \in X:\ f(x) \neq g(x)\}| \leq \epsilon |X|$.

Let $S_n$ denote the symmetric group of order $n$, the group of all permutations of $[n] := \{1,2,\ldots,n\}$. For $i,j \in [n]$, we let
$$T_{ij} = \{ \sigma \in S_n\colon \sigma(i)=j \}.$$
We call the $T_{ij}$'s the {\em 1-cosets} of $S_n$, since they are the cosets of stabilizers of points. 

Similarly, for $t>1$ and for two ordered $t$-tuples of distinct elements of $[n]$, $I=(i_1,\ldots, i_t)$ and $J=(j_1,\ldots, j_t)$, we let 
$$T_{IJ}= \{ \sigma \in S_n\colon \sigma(I)=J \} = \{\sigma \in S_n\colon \sigma(i_k)=j_k\ \forall k \in [t]\};$$
we call these the {\em $t$-cosets} of $S_n$. Abusing notation, we will sometimes use
$T_{ij}$ and $T_{IJ}$ to denote their own characteristic functions.

For any non-negative integer $t$, we let $U_t$ be the vector space of real-valued functions on $S_n$ whose Fourier transform is supported on irreducible representations of $S_n$ which are indexed by partitions of $n$ whose largest part has size at least $n-t$. In \cite{EFP} it is proved that $U_t$ is the space spanned by the $t$-cosets:

\begin{theorem}
\label{thm:characterization}
If $n$ and $t$ are integers with $0 \leq t \leq n$ then
$$U_t = \Span\{T_{IJ}: I,J \textrm{ are ordered }t\textrm{-tuples of distinct elements of }[n]\}.$$
\end{theorem}

If $t$ is fixed and $n$ is large, the irreducible representations of $S_n$ which are indexed by partitions of $n$ with largest part of size $t$ have dimension $\Theta(n^{t})$. The functions in $U_t$ therefore have Fourier transform supported on irreducible representations of dimension $O(n^t)$, and can be viewed as having low complexity. If $f$ is a real-valued function on $S_n$, we define the {\em degree} of $f$ to be the minimum $t$ such that $f \in U_t$. This can be seen as a measure of the complexity of $f$, analogous to the degree of a Boolean function on $\{0,1\}^n$. Indeed, by Theorem \ref{thm:characterization}, it is precisely the minimum possible total degree of a polynomial in the $T_{ij}$'s which is equal to $f$. 

The following theorem from \cite{EFP} characterizes the Boolean functions in $U_1$.
\begin{theorem}[Ellis, Friedgut, Pilpel]
\label{thm:char}
Let $f\colon S_n \rightarrow \{0,1\}$ be in $U_1$. Then $f$ is the characteristic function of a disjoint union of $1$-cosets.
\end{theorem} 

\noindent Note that a disjoint union of 1-cosets is necessarily of the form
$$\bigcup_{j \in S} T_{ij}$$
for some $i \in [n]$ and some $S \subset [n]$, or of the form
$$\bigcup_{i \in S}T_{ij}$$
for some $j \in [n]$ and some $S \subset [n]$, so is determined by the image or the pre-image of a fixed element. We call such a family a {\em dictatorship}.

\begin{comment}
By contrast, a disjoint union of $t$-cosets need not be determined by the image or the preimage of a fixed $t$-tuple. Consider, for example, the family
$$T_{(1,2)(1,2)} \cup T_{(1,3)(2,3)}\cup T_{(1,4)(3,4)} \cup T_{(1,5)(4,5)}\cup T_{(1,6)(5,6)} \cup \ldots \cup T_{(1,n)(n-1,n)},$$
which is a disjoint union of 2-cosets, but is not even determined by the images or pre-images of a bounded number of 2-tuples.
\end{comment}

In \cite{EFF1}, we proved that a Boolean function of expectation $O(1/n)$, whose Fourier transform is highly concentrated on irreducible representations corresponding to the partitions $(n)$ and $(n-1,1)$, is close in structure to a union of 1-cosets. Put another way, a Boolean function of expectation $O(1/n)$, which is close to $U_1$ (in Euclidean distance), has small symmetric difference with a union of 1-cosets. This is not true stability, as the Boolean function corresponding to $T_{11}\cup T_{22}$ is $O(1/n^2)$-close to $U_1$, but is not $1/(2n)$-close to any dictatorship, whereas a Boolean function {\em in} $U_1$ must be a dictatorship. We call it a `quasi-stability' result. (In \cite{EFF2}, on the other hand, we prove that a Boolean function of expectation bounded away from 0 and 1, which is close to $U_1$, must be close in structure to a dictatorship; this is `genuine' stability.)

Our aim in this paper is to prove an analogue of our quasi-stability result in \cite{EFF1}, for Boolean functions of degree at most $t$. Namely, we show that a Boolean function on \(S_n\) with expectation $O(n^{-t})$, whose Fourier transform is highly concentrated on irreducible representations indexed by partitions of $n$ with first row of length at least \(n-t\), is close in structure to a union of cosets of stabilizers of \(t\)-tuples. Put another way, a Boolean function of expectation $O(n^{-t})$, which is close to $U_t$ (in Euclidean distance), has small symmetric difference with a union of $t$-cosets. Here is the precise statement, our main theorem. 

\begin{theorem}
\label{thm:main}
For each \(t \in \mathbb{N}\), there exists \(C_t >0\) such that the following holds. Let $n \in \mathbb{N}$ with $n \geq t$, and let \(\mathcal{A} \subset S_n\) with \(|\mathcal{A}| = c(n-t)!\), where $c \geq 0$. Let \(f = 1_{\mathcal{A}} \colon S_n \to \{0,1\}\) denote the characteristic function of \(\mathcal{A}\), so that \(\mathbb{E}[f] = c/(n)_t\). Let $f_t$ denote the orthogonal projection of $f$ onto $U_t$. Let $0 \leq \epsilon \leq 1$. If \(\mathbb{E}[(f-f_t)^2] \leq \epsilon c/(n)_t\), then there exists $\mathcal{C} \subset S_n$ such that $\mathcal{C}$ is a union of \(\round c\) t-cosets of \(S_n\), and
\begin{equation}\label{eq:main-bound} |\mathcal{A} \triangle \mathcal{C}| \leq C_t(\epsilon^{1/2} + c/\sqrt{n})|\mathcal{A}|.\end{equation}
Moreover, we have $|c-\round c| \leq C_t(\epsilon^{1/2} + c/\sqrt{n})c.$
\end{theorem}
Here, $(n)_t : = n(n-1)\ldots(n-t+1)$ denotes the falling factorial, and if $c \geq 0$, $\round c$ denotes the closest integer to $c$, rounding up if $c+\tfrac{1}{2} \in \mathbb{N}$. The term `orthogonal projection onto $U_t$' will be made precise in section 2, where we define the standard inner product on the space of complex-valued functions on $S_n$.

Theorem \ref{thm:main} can be viewed as a non-Abelian analogue of the theorems in \cite{Bourgain} and \cite{KindlerSafra}, which concern Boolean functions on $\{0,1\}^n$ whose Fourier transform is highly concentrated on small sets. In \cite{KindlerSafra}, for example, it is shown that a Boolean function on $\{0,1\}^n$ whose Fourier transform is highly concentrated on sets of size at most $t$, must be close in structure to a junta depending upon at most $j(t)$ coordinates, for some function $j:\mathbb{N} \to \mathbb{N}$.

We remark that if $t \geq 2$, then a Boolean function in $U_t$ is not necessarily the characteristic function of a union of $t$-cosets. Theorem 27 in \cite{EFP} states that a Boolean function in $U_t$ is the characteristic function of a disjoint union of $t$-cosets, but this is false; a counterexample, and the error in the proof, is pointed out by the second author in \cite{F-note}. A counterexample when $t=2$ is as follows. Let $n \geq 8$. For any permutation $\sigma \in S_n$, define $x = x(\sigma) \in \{0,1\}^4$ by $x_i = 1\{\sigma(i) \in [4]\}$, and consider the function
\begin{equation}\label{eq:ce} f:S_n \to \{0,1\}; \quad \sigma \mapsto 1\{x_1 \geq x_2 \geq x_3 \geq x_4 \text{ or } x_1 \leq x_2 \leq x_3 \leq x_4\}.\end{equation}
It can be checked that $f \in U_2$, but the value of $f$ clearly cannot be determined by fixing the images of at most two numbers, so neither $f$ nor $1-f$ is a union of 2-cosets. It is easy to use $f$ to construct a counterexample for each $t \geq 3$, by considering a product of $f$ with the characteristic function of the pointwise stabilizer of a $(t-2)$-set. We note that the main application of Theorem 27 in \cite{EFP} was to characterize (for large $n$) the $t$-intersecting families in $S_n$ of maximum size (i.e., to characterize the cases of equality in the Deza-Frankl conjecture); fortunately, this characterization follows immediately e.g.\ from the Hilton-Milner type result of the first author in \cite{tstability}, where the proof does not depend on Theorem 27 in \cite{EFP} (and indeed predates the latter). \label{exp}

The example (\ref{eq:ce}) above shows that for $t \geq 2$, the right-hand side of (\ref{eq:main-bound}) must necessarily contain a term which does not tend to zero as $\epsilon \to 0$. (This contrasts with the situation when $t=1$; as mentioned above, any Boolean function in $U_1$ is a characteristic function of a disjoint union of 1-cosets.) We believe that the $\sqrt{n}$ factor in Theorem \ref{thm:main} is not sharp, but it suffices for our applications, and for the sake of brevity we have not attempted to improve it. 

\begin{comment}
On the other hand, taking $f$ as in (\ref{eq:ce}), we have $\mathbb{E}[1-f] = 2/n+O(1/n^2)$, and $1-f$ is $O(1/n^3)$ close to $1\{x_2 = 1 \text{ or } x_3 = 1 \text{ or } x_1=x_4=1\}$, which is a union of 2-cosets; this is consistent with the statement of Theorem \ref{thm:main}.
\end{comment}

Our proof of Theorem \ref{thm:main} is similar in some respects to the proof of the `quasi-stability' theorem in \cite{EFF1}, but the representation-theoretic tools are significantly more involved. It also turns out to be easier to deal with fourth moments (rather than third moments, as in \cite{EFF1}); these have the advantage of always being non-negative, although we pay the price of having to substitute approximations for many of the exact expressions in \cite{EFF1}.

We use one of our representation-theoretic lemmas to obtain an edge-isoperimetric inequality for the {\em transposition graph} on \(S_n\) (i.e. the Cayley graph on $S_n$ generated by all transpositions). This inequality, Theorem \ref{thm:approxiso}, is a little technical, so we defer its statement until section 5. It is asymptotically sharp for subsets of $S_n$ of size $n!/\textrm{poly}(n)$. We combine it with Theorem \ref{thm:main} and some additional combinatorial arguments to prove the following.

\begin{theorem}
\label{thm:iso}
Let \(\mathcal{A} \subset S_n\) with \(|\mathcal{A}| = (n-t)!\). If \(n\) is sufficiently large depending on \(t\), then
$$|\partial A| \geq |\partial T_{(1,2,\ldots,t)(1,2,\ldots,t)}|.$$
Equality holds if and only if \(\mathcal{A}\) is a \(t\)-coset of $S_n$.
\end{theorem}

Here, $\partial \mathcal{A}$ denotes the {\em edge-boundary} of $\mathcal{A}$ in the transposition graph, i.e. the set of all edges of the transposition graph which go between $\mathcal{A}$ and $S_n \setminus \mathcal{A}$. Theorem \ref{thm:iso} confirms a conjecture of Ben Efraim \cite{benefraim} in the relevant cases.

We note that some of the methods in this paper bear a high level resemblance to methods used more recently by Dinur, Khot, Kindler, Minzer and Safra in \cite{dkkms} to obtain structural results on subsets of the Grassmann graph with small edge expansion.

The remainder of this paper is structured as follows. In section 2, we give some background on general representation theory, on the representation theory of $S_n$, and on normal Cayley graphs. In section 3, we prove some preliminary representation-theoretic results we will need. In section 4, we prove our `quasi-stability' result, Theorem \ref{thm:main}. In section 5, we prove our isoperimetric inequality for the transposition graph on $S_n$, Theorem \ref{thm:iso}. In section 6, we outline how our results can be used to give an alternative proof of a stability result for $t$-intersecting families of permutations, obtained by the first author in \cite{tstability}. Finally, in section 7, we conclude with some open problems.

\section{Background}
\subsection*{Background on general representation theory}
In this section, we recall the basic notions and results we need from general representation theory. For more background, the reader may consult \cite{serre}.

Let \(G\) be a finite group. A {\em representation of \(G\) over \(\mathbb{C}\)} is a pair \((\rho,V)\), where \(V\) is a finite-dimensional complex vector space, and \(\rho\colon G \to GL(V)\) is a group homomorphism from \(G\) to the group of all invertible linear endomorphisms of \(V\). The vector space \(V\), together with the linear action of \(G\) defined by \(gv = \rho(g)(v)\), is sometimes called a \(\mathbb{C}G\)-{\em module}.  A {\em homomorphism} between two representations \((\rho,V)\) and \((\rho',V')\) is a linear map \(\phi\colon V \to V'\) such that such that \(\phi(\rho(g)(v)) = \rho'(g)(\phi(v))\) for all \(g \in G\) and \(v \in V\). If \(\phi\) is a linear isomorphism, the two representations are said to be {\em equivalent} or {\em isomorphic}, and we write \((\rho,V) \cong (\rho',V')\). If \(\dim(V)=n\), we say that \(\rho\) {\em has dimension} \(n\), and we write \(\dim(\rho) = n\).

The representation \((\rho,V)\) is said to be {\em irreducible} if it has no proper subrepresentation, i.e. there is no proper subspace of \(V\) which is \(\rho(g)\)-invariant for all \(g \in G\).

It turns out that for any finite group \(G\) there are only finitely many equivalence classes of irreducible representations of \(G\), and {\em any} representation of \(G\) is isomorphic to a direct sum of irreducible representations of \(G\). Hence, we may choose a set of representatives \(\mathcal{R}\) for the equivalence classes of irreducible representations of \(G\).

If \((\rho,V)\) is a representation of \(G\), the {\em character} \(\chi_{\rho}\) of \(\rho\) is the map defined by
\begin{eqnarray*}
\chi_{\rho}\colon  G & \to & \mathbb{C};\\
 \chi_{\rho}(g) &=& \textrm{Tr} (\rho(g)),
\end{eqnarray*}
where \(\textrm{Tr}(\alpha)\) denotes the trace of the linear map \(\alpha\) (i.e. the trace of any matrix of \(\alpha\)). Note that \(\chi_{\rho}(\textrm{Id}) = \dim(\rho)\), and that \(\chi_{\rho}\) is a {\em class function} on \(G\) (meaning that it is constant on each conjugacy-class of \(G\).)

The usefulness of characters lies in the following.
\begin{fact}
Two representations of \(G\) are isomorphic if and only if they have the same character. The irreducible characters form an orthonormal basis for the vector space of complex-valued class functions on \(G\).
\end{fact}

We now define the Fourier transform of a function on an arbitrary finite group.

\begin{definition}
Let \(\mathcal{R}\) be a complete set of non-isomorphic, irreducible representations of \(G\), i.e.\ containing one representative from each isomorphism class of irreducible representations of \(G\). Let \(f\colon G \to \mathbb{C}\) be a complex-valued function on \(G\). We define the {\em Fourier transform} of \(f\) at an irreducible representation \(\rho \in \mathcal{R}\) as
\begin{equation}
 \Hat{f}(\rho) = \frac {1}{|G|} \sum_{\sigma \in G} f(\sigma)\rho(\sigma);
\end{equation}
Note that \(\Hat{f}(\rho)\) is a linear endomorphism of \(V\).
\end{definition}

Let \(G\) be a finite group. Let \(\mathbb{C}[G]\) denote the vector space of all complex-valued functions on \(G\). Let \(\mathbb{P}\) denote the uniform probability measure on \(G\),
\[\mathbb{P}(\mathcal{A}) = |\mathcal{A}|/|G|\quad (\mathcal{A} \subset G),\]
let
\[\langle f,g \rangle = \frac{1}{|G|}\sum_{\sigma \in G} f(\sigma) \overline{g(\sigma)}\]
denote the corresponding inner product, and let
\[\|f\|_2 = \sqrt{\mathbb{E}[f^2]} = \sqrt{\frac{1}{|G|}\sum_{\sigma \in G} |f(\sigma)|^2}\]
denote the induced Euclidean norm.

For each irreducible representation \(\rho \in \mathcal{R}\), let
\[U_{\rho} := \{f \in \mathbb{C}[G] : \hat{f}(\pi) = 0 \textrm{ for all } \pi \in \mathcal{R} \setminus \{\rho\}\}.\]

We refer to this as the subspace of functions whose Fourier transform is supported on the irreducible representation \(\rho\), and we refer to the \(U_{\rho}\) as the {\em isotypical subspaces}. It turns out that the \(U_{\rho}\) are pairwise orthogonal, and that
\begin{equation}\label{eq:directsum}\mathbb{C}[G] = \bigoplus _{\rho \in \mathcal{R}} U_{\rho}.\end{equation}

For each \(\rho \in \mathcal{R}\), let \(f_{\rho}\) denote the orthogonal projection of \(f\) onto the subspace \(U_{\rho}\). It follows from the above that
\begin{equation}\label{eq:parseval}\|f\|_2^2 = \sum_{\rho \in \mathcal{R}} \|f_\rho\|_2^2.\end{equation}

The {\em group algebra} \(\mathbb{C}G\) denotes the complex vector-space with basis \(G\) and multiplication defined by extending the group multiplication linearly. In other words,
\[\mathbb{C}G = \left\{\sum_{g \in G}x_{g}g : x_{g} \in \mathbb{C}\ \forall g \in G\right\},\]
and
\[\left(\sum_{g \in G} x_{g}g\right)\left(\sum_{h\in G}y_{h}h\right) = \sum_{g,h \in G} x_{g}y_{h} (g h).\]
As a vector space, \(\mathbb{C}G\) may be identified with \(\mathbb{C}[G]\), by identifying \(\sum_{g \in G} x_g g\) with the function \(g \mapsto x_g\).

\subsection*{Background on the representation theory of \(S_n\)}
Our treatment below follows \cite{JamesKerber} and \cite{sagan}.
\begin{definition}
A \emph{partition} of \(n\) is a non-increasing sequence of positive integers summing to \(n\), i.e. a sequence $\lambda = (\lambda_1, \ldots, \lambda_k)$ with \(\lambda_{1} \geq \lambda_{2} \geq \ldots \geq \lambda_{k} \geq 1\) and \(\sum_{i=1}^{k} \lambda_{i}=n\); we write \(\lambda \vdash n\). For example, \((3,2,2) \vdash 7\).
\end{definition}
If a part $a$ appears $t$ times in a partition then we write $a^t$ instead of $\overbrace{a,\ldots,a}^{t \text{ times}}$. For example, $(3,2^2) = (3,2,2)$.

The following two orders on partitions of \(n\) will be useful.

\begin{definition}
  (Dominance order) Let $\lambda = (\lambda_1, \ldots, \lambda_k)$ and $\mu = (\mu_1,
  \ldots, \mu_l)$ be partitions of $n$. We say that $\lambda \domgeq
  \mu$ ($\lambda$ {\em dominates} $\mu$) if $\sum_{j=1}^{i} \lambda_i \geq \sum_{j=1}^{i} \mu_i \ \forall i$ (where we define \(\lambda_{i} = 0 \ \forall i > k,\ \mu_{i} = 0 \ \forall i > l\)).
\end{definition}

It is easy to see that this is a partial order. 

\begin{definition}
  (Lexicographic order) Let $\lambda = (\lambda_1, \ldots, \lambda_r)$ and $\mu = (\mu_1,
  \ldots, \mu_s)$ be partitions of $n$. We say that $\lambda > \mu$ if \(\lambda_{j} > \mu_{j}\), where \(j = \min\{i \in [n]\colon \lambda_{i} \neq \mu_{i}\}\).
\end{definition}
It is easy to see that this is a total order which extends the dominance order.

It is well-known that there is an explicit 1-1 correspondence between irreducible representations of \(S_{n}\) (up to isomorphism) and partitions of \(n\). The reader may refer to \cite{sagan} for a full description of this correspondence, or to the paper \cite{EFP} for a shorter description.

For each partition \(\alpha\) of \(n\), we write \([\alpha]\) for the corresponding isomorphism class of irreducible representations of \(S_n\), and we write \(U_{\alpha} = U_{[\alpha]}\) for the vector space of complex-valued functions on \(S_n\) whose Fourier transform is supported on \([\alpha]\). Similarly, if \(f \in \mathbb{C}[S_n]\), we write \(f_{\alpha}\) for the orthogonal projection of \(f\) onto \(U_{\alpha}\). By (\ref{eq:directsum}), we have
\[\mathbb{C}[S_n] = \bigoplus_{\alpha \vdash n} U_{\alpha},\]
and the \(U_{\alpha}\) are pairwise orthogonal. For any $f \in \mathbb{C}[S_n]$, we have
\[f = \sum_{\alpha \vdash n} f_{\alpha},\]
and
$$\|f\|_2^2 = \sum_{\alpha \vdash n} \|f_{\alpha}\|^2.$$
We write \(\chi_{\alpha}\) for the character of the irreducible representation corresponding to \(\alpha\), and we write \(\dim[\alpha]\) for the dimension of \([\alpha]\).

\begin{definition} If \(\lambda = (\lambda_1,\ldots,\lambda_l)\) is a partition of \(n\), the {\em Young diagram of shape \(\lambda\)} is an array of left-justified cells with \(\lambda_i\) cells in row \(i\), for each \(i \in [l]\).
\end{definition}
For example, the Young diagram of the partition \((3,2^{2})\) is:
  \[
\yng(3,2,2)
\]
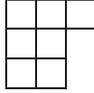
\begin{definition}
A {\em \(\lambda\)-tableau} is a Young diagram of shape \(\lambda\), each of whose cells contains a number between 1 and \(n\). If \(\mu\) is a partition of \(n\), a Young tableau is said to have {\em content \(\mu\)} if it contains \(\mu_i\) \(i\)'s for each \(i \in \mathbb{N}\).
\end{definition}
\begin{definition}
A Young tableau is said to be {\em standard} if it has content \((1,1,\ldots,1)\) and the numbers are strictly increasing down each row and along each column.
\end{definition}
\begin{definition}
A Young tableau is said to be {\em semistandard} if the numbers are non-decreasing along each row and strictly increasing down each column.
\end{definition}
The relevance of standard Young tableaux stems from the following
\begin{theorem}
\label{thm:dimension}
If \(\alpha\) is a partition of \(n\), then \(\dim[\alpha]\) is the number of standard \(\alpha\)-tableaux.
\end{theorem}
We say that two \(\lambda\)-tableaux of content $(1,1,\ldots,1)$ are {\em row-equivalent} if they contain the same set of numbers in each row. A row-equivalence-class of \(\lambda\)-tableaux with content $(1,1,\ldots,1)$ is called a {\em \(\lambda\)-tabloid}. Consider the natural left action of \(S_n\) on the set of \(\lambda\)-tabloids, i.e. if a $\lambda$-tabloid $\mathcal{T}$ has $i$th row $R_i \in [n]^{(\lambda_i)}$ for each $i$, then $\sigma(\mathcal{T})$ has $i$th row $\sigma(R_i)$ for each $i$. Let \(M^{\lambda}\) denote the induced permutation representation. We write \(\xi_{\lambda}\) for the character of \(M^{\lambda}\); the \(\xi_{\lambda}\) are called the {\em permutation characters} of \(S_n\).

{\em Young's theorem} gives the decomposition of each permutation representation into irreducible representations of \(S_n\), in terms of the {\em Kostka numbers}.

\begin{definition}
Let \(\lambda\) and \(\mu\) be partitions of \(n\). The {\em Kostka number} \(K_{\lambda,\mu}\) is the number of semistandard \(\lambda\)-tableaux of content \(\mu\).
\end{definition}

\begin{theorem}[Young's theorem]
\label{thm:young}
If \(\mu\) is a partition of \(n\), then
\[M^{\mu} \cong \bigoplus_{\substack{\lambda \vdash n\colon\\ \lambda \geq \mu\hphantom{\colon}}} K_{\lambda, \mu}
  [\lambda].\]
\end{theorem}

It follows that for each partition \(\mu\) of \(n\), we have
\[\xi_{\mu} = \sum_{\substack{\lambda \vdash n\colon\\ \lambda \geq \mu\hphantom{\colon}}} K_{\lambda,\mu} \chi_{\lambda}.\]

On the other hand, we can express the irreducible characters in terms of the permutation characters using the {\em determinantal formula}: for any partition \(\alpha\) of \(n\),
\begin{equation}\label{eq:determinantalformula} \chi_{\alpha} = \sum_{\pi \in S_{n}} \sgn(\pi) \xi_{\alpha - \textrm{id}+\pi}.\end{equation}
Here, if \(\alpha = (\alpha_{1},\alpha_{2},\ldots,\alpha_{l})\), \(\alpha - \textrm{id}+\pi\) is defined to be the sequence
\[(\alpha_{1}-1+\pi(1),\alpha_{2}-2+\pi(2),\ldots,\alpha_{l}-l+\pi(l)).\]
If this sequence has all its entries non-negative, we let \(\overline{\alpha-\textrm{id}+\pi}\) be the partition of \(n\) obtained by reordering its entries, and we define \(\xi_{\alpha - \textrm{id}+\pi} = \xi_{\overline{\alpha-\textrm{id}+\pi}}\). If the sequence has a negative entry, we define \(\xi_{\alpha - \textrm{id}+\pi} = 0\). Note that if \(\xi_{\beta}\) appears on the right-hand side of (\ref{eq:determinantalformula}), then \(\beta \geq \alpha\), so the determinantal formula expresses \(\chi_{\alpha}\) in terms of \(\{\xi_{\beta}: \ \beta \geq \alpha\}\). (See \cite{JamesKerber} for a proof of (\ref{eq:determinantalformula}).)

For each \(t \in \mathbb{N}\), we define
\[U_t = \bigoplus_{\substack{\alpha \vdash n\colon\\ \alpha_1 \geq n-t}} U_{\alpha},\]
i.e. \(U_t\) is the subspace of functions on \(S_n\) whose Fourier transform is concentrated on irreducible representations corresponding to partitions with first row of length at least \(n-t\). As mentioned in the introduction, it was proved in \cite{EFP} that $U_t$ is the linear span of the characteristic functions of the $t$-cosets of $S_n$:
\[U_t = \Span\{T_{IJ} : I,J\ \textrm{are ordered }t\textrm{-tuples of distinct elements of }[n]\}.\]
For brevity, we shall sometimes write $\mathcal{C}(n,t)$ for the set of $t$-cosets of $S_n$.

If \(f\colon S_n \to \mathbb{C}\), we write \(f_t\) for the orthogonal projection of \(f\) onto \(U_t\), for each \(t \in [n]\); equivalently,
\[f_t = \sum_{\substack{\alpha \vdash n\colon\\ \alpha_1 \geq n-t}} f_{\alpha}.\]

For each \(t \in \mathbb{N}\), we write
\[V_t = \bigoplus_{\substack{\alpha \vdash n\colon\\ \alpha_1 = n-t}} U_{\alpha}.\]
Note that the \(V_t\) are pairwise orthogonal, and that
\[U_t = U_{t-1} \oplus V_t\]
for each \(t \leq n\).
\begin{comment}
Recall the following theorem from \cite{EFP}, which completely characterizes the Boolean functions in \(U_t\).
\begin{theorem}[Bennabbas, Friedgut, Pilpel]
\label{thm:characterization}
If \(\mathcal{A} \subset S_n\) has \(1_{\mathcal{A}} \in U_t\), then \(\mathcal{A}\) is a disjoint union of \(t\)-cosets of \(S_n\).
\end{theorem}

Our aim is to prove that if \(f\colon S_n \to \{0,1\}\) is a Boolean function on \(S_n\) such that
\[\mathbb{E}[(f - f_t)^2]\]
is small, then there exists a Boolean function \(F\) such that
\[\mathbb{E}[(f-F)^2]\]
is small, and \(F\) is the characteristic function of a union of \(t\)-cosets of \(S_n\). Note that, by (\ref{eq:parseval}), we have
\[\mathbb{E}[(f-f_t)^2] = \|f-f_t\|_2^2 = \sum_{\alpha_1 < n-t} \|f_\alpha\|^2 = \textrm{dist}(f,U_t)^2,\]
where \(\textrm{dist}\) denotes the Euclidean distance.
\end{comment}
\subsection*{Background on normal Cayley graphs}
Several of our results will involve the analysis of {\em normal Cayley graphs} on \(S_n\). We recall the definition of a Cayley graph on a finite group.
\begin{definition}
Let \(G\) be a finite group, and let \(S \subset G \setminus \{\textrm{Id}\}\) be inverse-closed (meaning that \(S^{-1}=S\)). The (left) {\em Cayley graph on \(G\) with generating set \(S\)} is the graph with vertex-set \(G\), where we join \(g\) to \(sg\) for every \(g \in G\) and \(s \in S\); we denote it by \(\Cay(G,S)\). Formally,
\[V(\Cay(G,S)) = G,\quad E(\Cay(G,S)) = \{\{g,sg\}\colon g \in G,\ s \in S\}.\]
Note that the Cayley graph \(\Cay(G,S)\) is \(|S|\)-regular. If the generating set \(S\) is conjugation-invariant, i.e. is a union of conjugacy classes of \(G\), then the Cayley graph \(\Cay(G,S)\) is said to be a {\em normal} Cayley graph.
\end{definition} 
The connection between normal Cayley graphs and representation theory arises from the following fundamental theorem, which states that for any normal Cayley graph, the eigenspaces of its adjacency matrix are in 1-1 correspondence with the isomorphism classes of irreducible representations of the group.

\begin{theorem}[Frobenius / Schur / Diaconis-Shahshahani]
\label{thm:normalcayley}
Let \(G\) be a finite group, let \(S \subset G\) be an inverse-closed, conjugation-invariant subset of \(G\), let \(\Gamma = \Cay(G,S)\) be the Cayley graph on \(G\) with generating set \(S\), and let \(A\) be the adjacency matrix of \(\Gamma\). Let \(\mathcal{R}\) be a complete set of non-isomorphic complex irreducible representations of \(G\). Then we have
\[\mathbb{C}[G] = \bigoplus_{\rho \in \mathcal{R}}U_{\rho},\]
and each \(U_{\rho}\) is an eigenspace of \(A\) with dimension \(\dim(\rho)^{2}\) and eigenvalue
\begin{equation}\label{eq:normalcayley}\lambda_{\rho} = \frac{1}{\dim(\rho)}\sum_{g \in S} \chi_{\rho}(g).\end{equation}
\end{theorem}

Note that the adjacency matrix \(A\) in the above theorem is given by
\[A_{g,h} = 1_S(g h^{-1})\quad (g,h \in G).\]

We will need a slight generalization of Theorem \ref{thm:normalcayley}. Recall that if \(G\) is a finite group, a {\em class function} on \(G\) is a function on \(G\) which is constant on each conjugacy-class of \(G\). If \(S\) is conjugation-invariant, then \(1_S\) is clearly a class function. Observe that if \(A\) is any matrix in \(\mathbb{C}[G \times G]\) defined by
\[A_{g,h} = w(gh^{-1}),\]
where \(w\) is a class function on \(G\), then \(A\) is a linear combination of the adjacency matrices of normal Cayley graphs. Therefore, the isotypical subspaces \(U_\rho\) are again eigenspaces of \(A\), and the corresponding eigenvalues are given by
\begin{equation}\label{eq:wnormalcayley}\lambda_{\rho} = \frac{1}{\dim(\rho)}\sum_{g \in S} w(g) \chi_{\rho}(g).\end{equation}

Applying Theorem \ref{thm:normalcayley} to \(S_n\), we see that if \(\Gamma = \Cay(S_n,X)\) is a normal Cayley graph on \(S_n\) and \(A\) is the adjacency matrix of \(\Gamma\) then we have
\[\mathbb{C}[S_n] = \bigoplus_{\alpha \vdash n}U_{\alpha},\]
and each \(U_{\alpha}\) is an eigenspace of \(A\) with dimension \(\dim[\alpha]\) and eigenvalue
\begin{equation}\label{eq:normalcayleygraph} \lambda_{\alpha} = \frac{1}{\dim[\alpha]}\sum_{\sigma \in X} \chi_{\alpha}(\sigma).\end{equation}

\section{Preliminary results}
In this section, we prove some preliminary representation-theoretic lemmas.
 
\subsection*{Preliminary results on the characters of \(S_n\)}
In this section, we prove some easy results on the characters of \(S_n\). We will use the following standard notation. For \(n,r \in \mathbb{N}\), we will write 
\[(n)_r := n(n-1)\ldots(n-r+1)\]
for the \(r\)th falling factorial moment of \(n\), and we will write $([n])_{r}$ for the set of all $r$-tuples of distinct elements of $[n]$.

If \(f = f(n,t),\ g = g(n,t):\mathbb{N} \times \mathbb{N} \to \mathbb{R}_{\geq 0}\), we will write \(f = O_t(g)\) if for every \(t \in \mathbb{N}\), there exists \(C_t >0\) such that for all $n \in \mathbb{N}$, \(f(n,t) \leq C_t g(n,t)\). Similarly, we write $f = \Omega_t(g)$ if for every $t \in \mathbb{N}$, there exists \(c_t >0\) such that for all $n \in \mathbb{N}$, \(f(n,t) \geq c_t g(n,t)\). We write $f = \Theta_t(g)$ if $f = O_t(g)$ and $f = \Omega_t(g)$.

We begin with a crude upper bound on the dimensions of the irreducible characters \(\chi_{\alpha}\).

\begin{lemma}
\label{lemma:dimension}
Let \(\alpha\) be a partition of \(n\) with \(\alpha_1 = n-s\). Then
\[\dim[\alpha] \leq (n)_s.\]
\end{lemma}
\begin{proof}
Since \(\alpha_1 = n-s\), we have \(\alpha \geq (n-s,1^s)\). Recall (or observe from Theorem \ref{thm:young}) that for each partition \(\beta \geq (n-s,1^s)\), \([\beta]\) is an irreducible constituent of \(M^{(n-s,1^s)}\), which has dimension \((n)_s\). It follows that \(\dim[\alpha] \leq (n)_s\), as required.
\end{proof}

We also need crude bounds on the Kostka numbers \(K_{\alpha,(n-s,1^s)}\).

\begin{lemma}
\label{lemma:kostkaestimate}
Let \(\alpha\) be a partition of \(n\) with \(\alpha_1 = n-s\). Then
\[\binom{n-s}{s} K_{\alpha,(n-s,1^s)} \leq \dim[\alpha] \leq \binom{n}{s} K_{\alpha,(n-s,1^s)}.\]
\end{lemma}
\begin{proof}
Let us write $\alpha = (n-s,\gamma)$, where $\gamma$ is a partition of $s$. Recall from Theorem \ref{thm:dimension} that \(\dim[\alpha]\) is the number of standard \(\alpha\)-tableaux. Observe that we may construct \(\binom{n-s}{s} \dim[\gamma]\) distinct standard \(\alpha\)-tableaux as follows:
\begin{enumerate}
\item Place \(i\) in cell \((1,i)\) for \(i=1,2,\ldots,s\).
\item Choose any \(s\)-set \(\{i_1,\ldots,i_s\}\) from \(\{s+1,\ldots,n\}\) (\(\binom{n-s}{s}\) choices).
\item Place \(\{i_1,\ldots,i_s\}\) in the cells below row 1 so that they are strictly increasing along each row and down each column. The number of ways of doing this is precisely the number of standard $\gamma$-tableaux, which is $\dim[\gamma]$.
\item Place the other numbers in the remaining cells of row 1, in increasing order from left to right.
\end{enumerate}
Hence, there are at least $\binom{n-s}{s}\dim[\gamma]$ standard $\alpha$-tableaux. On the other hand, removing the first row from any standard \(\alpha\)-tableau produces a standard \(\gamma\)-tabeau (filled with some \(s\) of the numbers between \(1\) and \(n\), rather than with $\{1,2,\ldots,s\}$). Hence, the number of standard \(\alpha\)-tableaux is at most \(\binom{n}{s} \dim[\gamma]\). Therefore,
$$\binom{n - s}{s} \dim[\gamma] \leq \dim[\alpha] \leq \binom{n}{s} \dim[\gamma].$$
Recall from Theorem \ref{thm:young} that \(K_{\alpha,(n-s,1^s)}\) is the number of semistandard \(\alpha\)-tableaux of content \((n-s,1^s)\). A semistandard \(\alpha\)-tableau of content \((n-s,1^s)\) must have all \(n-s\) of its 1's in the first row; after deleting the first row, what is left is precisely a standard $\gamma$-tableau, filled with the numbers $\{2,3,\ldots,s+1\}$, rather than with $\{1,2,\ldots,s\}$. Hence, the number of semistandard \(\alpha\)-tableaux of content \((n-s,1^s)\) is precisely the number of standard $\gamma$-tableaux, so $K_{\alpha,(n-s,1^s)} = \dim[\gamma]$. Therefore,
$$\binom{n-s}{s} K_{\alpha,(n-s,1^s)} \leq \dim[\alpha] \leq \binom{n}{s} K_{\alpha,(n-s,1^s)},$$
as required.
\end{proof}

Finally, we need a crude lower bound on the \(L^1\)-norm of the characters of the symmetric group.

\begin{lemma}
\label{lemma:L1norm}
For each \(s \in \mathbb{N}\), there exists \(K_s >0\) such that for any $n \in \mathbb{N}$ and any partition \(\alpha\) of \(n\) with \(\alpha_1 \geq n-s\), we have
\[\frac{1}{n!} \sum_{\sigma \in S_n} |\chi_{\alpha}(\sigma)| \geq K_s.\]
\end{lemma}
To assist with the proof, we introduce the following notation/definition.
\begin{definition}
Let \(s \in \mathbb{N}\), and let \(\beta = (\beta_1,\ldots,\beta_l)\) be a partition of \(n\) with \(\beta_1 \geq n-s\). We define \(X_{\beta,s}\) to be the set of all permutations with cycle-type in the set
\[\{(\lambda_1,\ldots,\lambda_k,\beta_2,\ldots,\beta_l) : (\lambda_1,\ldots,\lambda_k) \vdash \beta_1,\ |\lambda_i| > s\ \forall i \in [k]\}.\]
\end{definition}

\begin{definition} For $m,s \in \mathbb{N}$ with \(m>s\), we let \(N_{m,s}\) denote the number of permutations in \(S_m\) whose cycles all have lengths greater than \(s\). 
\end{definition}

We also need three preparatory claims.
\begin{claim}
\label{claim:cycles}
For any $m,s \in \mathbb{N}$ with $m>s$, we have $N_{m,s} \geq m!/(2s+1)$.
\end{claim}
\begin{proof}[Proof of Claim]
By induction on $m$, for each fixed $s \in \mathbb{N}$. If $s+1\leq m \leq 2s+1$, then $N_{m,s}$ is precisely the number of $m$-cycles in $S_m$, which is $(m-1)! \geq m!/(2s+1)$. We now derive a recurrence relation for $N_{m,s}$, for all $m \geq 2s+2$. Let $\sigma \in S_m$ be a permutation with all cycles of length greater than $s$. Let $i = \sigma(m)$; then $i <m$ and we may write $\sigma = (i\ m) \rho$, where $\rho \in S_{[m-1]}$ is a permutation with all its cycles of length greater than $s$, except possibly an $s$-cycle containing $i$. Conversely, given such a pair $(\rho,i)$, $(i\ m) \rho$ has all its cycles of length greater than $s$. It follows that
$$N_{m,s} = (m-1)(N_{m-1,s} + (m-2)(m-3)\ldots(m-s)N_{m-s-1,s})\quad \forall m > s,$$
where we define $N_{0,s}:=1$. We now perform the inductive step. Let $m \geq 2s+2$, and assume that the claim holds when $m$ is replaced by any $m' < m$. Since $m-1-s > s$, by the inductive hypothesis we have
\begin{align*} N_{m,s} & = (m-1)(N_{m-1,s} + (m-2)(m-3)\ldots(m-s)N_{m-s-1,s})\\
& \geq (m-1)((m-1)! + (m-2)(m-3)\ldots(m-s)(m-s-1)!)/(2s+1)\\
& = m!/(2s+1),
\end{align*}
completing the inductive step and proving the claim.
\end{proof}

\begin{claim}
Let \(\beta = (\beta_1,\ldots,\beta_l)\) be a partition of \(n\) with \(\beta_1 \geq n-s\), where \(n > 2s\). Then
\[|X_{\beta,s}| \geq L_s n!,\]
where \(L_s >0\) depends upon \(s\) alone.
\end{claim}
\begin{proof}[Proof of Claim]
Let $\beta_1=n-r$, where $r \leq s$. Observe that
\[|X_{\beta,s}| \geq \binom{n}{r} N_{n-r,s},\]
since there are \(\binom{n}{r}\) choices for the \(r\) numbers to go in the cycles of lengths \(\beta_2,\ldots,\beta_l\), and then \(N_{n-r,s}\) choices for placing the other \(n-r\) numbers in the cycles of lengths greater than \(s\). It follows, using Claim \ref{claim:cycles}, that
\[|X_{\beta,s}| \geq \binom{n}{r} (n-r)!/(2s+1) = \frac{n!}{r!(2s+1)} \geq \frac{n!}{s!(2s+1)} n! = L_s n!,\]
where \(L_s := 1/(s!(2s+1))\), proving the claim.
\end{proof}
\begin{claim}
If \(\alpha,\beta\) are partitions of \(n\) with \(\alpha_1, \beta_1 \geq n-s\), where $n > 2s$, then \(\chi_{\alpha}\) is constant on \(X_{\beta,s}\).
\end{claim}  
\begin{proof}[Proof of Claim.]
Recall from (\ref{eq:determinantalformula}) that \(\chi_{\alpha}\) can be expressed as a linear combination of the permutation characters \(\{\xi_{\gamma} : \gamma \geq \alpha\}\). Recall also that \(\xi_{\gamma}(\sigma)\) is simply the number of \(\gamma\)-tabloids fixed by \(\sigma\), which is the number of $\gamma$-tabloids that can be produced by taking each row to be a union of cycles of $\sigma$. This is clearly the same for all \(\sigma \in X_{\beta,s}\), since in order to produce a $\gamma$-tabloid from $\sigma \in X_{\beta,s}$ as above, the top row must contain the union of all the cycles of length greater than $s$. 

It follows that \(\chi_{\alpha}(X_{\beta,s}) = \chi_{\alpha}(\sigma_\beta)\), where \(\sigma_\beta\) is a permutation of cycle-type \(\beta\) --- so \(\chi_{\alpha}(X_{\beta,s})\) is simply the \((\alpha,\beta)\)-entry of the character table of \(S_n\).
\end{proof}
We can now prove the lemma.
\begin{proof}[Proof of Lemma \ref{lemma:L1norm}.]
Let $\alpha$ be a partition of $n$ with $\alpha_1 \geq n-s$. Recall that for any partition \(\alpha\) of \(n\), the top-left minor of the character table of \(S_n\) indexed by partitions \(\geq \alpha\), is invertible (see, for example, \cite[Theorem 20]{EFP}). It follows that it cannot have a row of zeros, so there exists \(\beta \geq \alpha\) such that \(\chi_{\alpha}(\sigma_\beta) \neq 0\). Since the irreducible characters of \(S_n\) are all integer-valued, we have \(|\chi_{\alpha}(\sigma_\beta)| \geq 1\). Hence, we have
\[\frac{1}{n!} \sum_{\sigma \in S_n} |\chi_{\alpha}(\sigma)| \geq \frac{1}{n!} |X_{\beta,s}| |\chi_{\alpha}(\sigma_\beta)| \geq L_s.\]
This proves the lemma when $n > 2s$. Since no character vanishes on all of $S_n$, by taking $K_s$ to be sufficiently small, we see that the lemma holds for all $n$.
\end{proof}

\subsection*{A preliminary result on the projections of Boolean functions}
The following result will be useful both in the proof of our quasi-stability theorem, and in that of our approximate isoperimetric inequality for the transposition graph.
\begin{lemma}
\label{lemma:smallnorm}
For each \(t \in \mathbb{N}\) there exists \(C'_t>0\) such that the following holds. If \(\mathcal{A} \subset S_n\) and \(f = 1_{\mathcal{A}}\) denotes the characteristic function of \(\mathcal{A}\), then
\[\|f_{t-1}\|_2^2 \leq C_t' n^{t-1} (|\mathcal{A}|/n!)^2.\]
\end{lemma}

\begin{proof}
Let \(\alpha\) be a partition of \(n\) with \(\alpha_1 = n-s\). Our aim is to prove that 
\[\|f_{\alpha}\|_2^2 \leq C_s'' n^s (|\mathcal{A}|/n!)^2,\]
where \(C_s''>0\) depends upon \(s\) alone. If \(s=0\) then \(\alpha = (n)\), and we have
\[\|f_{(n)}\|_2^2 = (|\mathcal{A}|/n!)^2,\]
so we may take \(C_0'' = 1\).

Suppose now that \(s \geq 1\). Let \(\Gamma = \Cay(S_n,X)\) be any normal Cayley graph on \(S_n\), and let \(A\) denote its adjacency matrix. If \(\mathcal{B} \subset S_n\), let \(e(\mathcal{B})\) denote the number of edges of \(\Gamma\) within \(\mathcal{B}\), and let \((\lambda_\beta)_{\beta \vdash n}\) denote the eigenvalues of \(A\). By Theorem \ref{thm:normalcayley}, we have
\[Af = \sum_{\beta \vdash n} \lambda_\beta f_{\beta},\]
and therefore
\[2e(\mathcal{A}) = n!\langle f,Af \rangle = n! \sum_{\beta \vdash n} \lambda_\beta \|f_{\beta}\|^2.\]
Trivially, we have \(e(\mathcal{A}) \leq\binom{|\mathcal{A}|}{2}\), and therefore
\[|\mathcal{A}|^2 \geq |\mathcal{A}|(|\mathcal{A}|-1) \geq 2e(\mathcal{A}) = n! \sum_{\beta \vdash n} \lambda_\beta \|f_{\beta}\|^2 \geq n! \lambda_{\alpha}\|f_\alpha\|_2^2.\]
It follows that
\[\|f_{\alpha}\|_2^2 \leq \frac{|\mathcal{A}|^2}{n!\lambda_{\alpha}}.\]
In order to minimize the right-hand side, we will choose \(\Gamma\) with \(\lambda_{\alpha}\) as large as possible.

By (\ref{eq:normalcayleygraph}), we have 
\[\lambda_{\alpha} = \frac{1}{\dim[\alpha]} \sum_{\sigma \in X} \chi_{\alpha}(\sigma).\]
To maximize $\lambda_{\alpha}$, we simply take \(X = \{\sigma \in S_n : \chi_{\alpha}(\sigma) > 0\}\). The corresponding Cayley graph \(\Cay(S_n,X)\) then has
\[\lambda_{\alpha} = \frac{1}{\dim[\alpha]}\sum_{\substack{\sigma \in S_n\colon\\ \chi_{\alpha}(\sigma) >0}} \chi_{\alpha}(\sigma).\]
Since \(\alpha \neq (n)\), by the orthogonality of the irreducible characters, we have
\[\langle \chi_{\alpha},\chi_{(n)} \rangle =0,\]
i.e.
\[\sum_{\sigma \in S_n} \chi_{\alpha}(\sigma) = 0,\]
and therefore
\[\sum_{\substack{\sigma \in S_n\colon\\ \chi_{\alpha}(\sigma)>0}} \chi_{\alpha}(\sigma) = \tfrac{1}{2} \sum_{\sigma \in S_n} |\chi_{\alpha}(\sigma)|.\]
Hence,
\[\lambda_{\alpha} = \frac{1}{2\dim[\alpha]} \sum_{\sigma \in S_n} |\chi_{\alpha}(\sigma)|.\]

We therefore obtain
\[\|f_{\alpha}\|_2^2 \leq \frac{2\dim[\alpha]|\mathcal{A}|^2}{n!\sum_{\sigma \in S_n}|\chi_{\alpha}(\sigma)|}.\]

By Lemma \ref{lemma:dimension}, we have \(\dim[\alpha] \leq (n)_{s} \leq n^s\), and by Lemma \ref{lemma:L1norm}, we have
\[\frac{1}{n!}\sum_{\sigma \in S_n} |\chi_{\alpha}(\sigma)| \geq K_{s}.\]
Hence,
\[\|f_{\alpha}\|_2^2 \leq \frac{2n^s}{K_{s}}(|\mathcal{A}|/n!)^2 = C_s'' n^s (|\mathcal{A}|/n!)^2,\]
where \(C_s'' = 2/K_s\).

Recall that
\[\|f_{t-1}\|_2^2 = \sum_{\substack{\alpha \vdash n\colon\\ \alpha_1 \geq n-t+1}} \|f_\alpha\|_2^2.\]

Note that for each \(n \geq 2s\), the number of partitions \(\alpha\) of \(n\) with \(\alpha_1 = n-s\) is equal to the number of partitions of \(s\). Hence, the number of terms in the above sum is bounded from above by a function of \(t\) alone, so there exists \(C_t' >0\) such that
\[\|f_{t-1}\|_2^2 \leq C_t' n^{t-1} (|\mathcal{A}|/n!)^2,\]
as required.
\end{proof}

Observe that Lemma \ref{lemma:smallnorm} implies that if \(\mathcal{A} \subset S_n\) with \(|\mathcal{A}| = o((n-t+1)!)\), then \(\|(1_{\mathcal{A}})_{t-1}\|_2^2 = o(|\mathcal{A}|/n!)\), i.e. the projection of \(1_{\mathcal{A}}\) onto \(U_{t-1}\) has small \(L^2\)-norm. In other words, the Fourier transform of \(1_{\mathcal{A}}\) has very little `mass' on the irreducible representations \(\{[\alpha] : \alpha_1 \geq n-t+1\}\). 

\section{Proof of the quasi-stability result}
Our aim in this section is to prove Theorem \ref{thm:main}. We first give an overview of the proof.

\begin{myproof}[Proof overview.]
The proof proceeds as follows. Instead of working with \(f_t\), the projection of \(f\) onto \(U_t\), it will be more convenient to work with two different functions, which are both close to \(f_t\) in \(L^2\)-norm. Namely, we will let
\[g_t = f_t - f_{t-1};\]
note that $g_t$ is the orthogonal projection of \(f\) onto \(V_t\), as $V_t$ is the orthogonal complement of $U_{t-1}$ in $U_t$. We will see that if \(c = o(n)\), then \(\|f_{t-1}\|_2^2 = o(\|f\|_2^2)\), i.e. the projection of \(f\) onto \(U_{t-1}\) is small, so \(g_t\) is close to \(f_t\) in \(L^2\)-norm. 

For each \(t\)-coset \(T\), we define
\[a_{T} = |\mathcal{A} \cap T|/(n-t)!;\]
note that \(a_T = (n)_t \langle f, 1_{T} \rangle = \mathbb{E}[f|_{T}]\), the expectation of \(f\) restricted to \(T\). Our aim is to show that there are approximately \(c\) \(t\)-cosets \(T\) such that \(a_T\) is close to 1. However, there are no simple relationships between the \(a_T\)'s and the moments of \(f_t\), so instead, we work with the quantities
\[b_{T} := (n)_t \langle g_t, 1_{T} \rangle = \mathbb{E}[g_{t}|_{T}].\]
We define
\[h_{t} := \sum_{T} b_{T} 1_{T}.\]
Because \(h_t\) is in \(V_t\) (as opposed to \(f_t\), which is in \(U_t\) but not in \(V_t\)), there is a close relationship between \(\mathbb{E}[h_t^2]\) and \(\sum_{T} b_T^2\), and between \(\mathbb{E}[h_t^4]\) and \(\sum_{T} b_T^4\). Moreover, we will also see that \(h_t\) is close to \(g_t\) (and therefore to \(f\)) in \(L^2\)-norm, so \(\mathbb{E}[h_t^2] \approx c/(n)_t\). This, together with the relationship between \(\mathbb{E}[h_t^2]\) and \(\sum_{T}b_T^2\), will imply that \(\sum_T b_T^2 \approx c\).

Using the fact that \(\mathbb{E}[(h_t - f)^2]\) is small, together with the fact that \(f\) is a Boolean function with expectation \(c/(n)_t\), we will show that \(\mathbb{E}[h_t^4] \approx c/(n)_t\). This, together with the relationship between \(\sum_T b_T^4\) and \(\mathbb{E}[h_t^4]\), will imply that \(\sum_T b_T^4 \approx c\). As long as \(c = o(n)\), we will have \(|b_T - a_T| = o(1)\) for each \(T\), and therefore \(b_T \leq 1+o(1)\). Hence, the only way we can have \(\sum_T b_T^2 \approx c\) and \(\sum_{T} b_T^4 \approx c\) is if approximately \(c\) of the \(b_T\)'s are close to 1. This in turn implies that the corresponding \(a_T\)'s are close to 1, completing the proof.

Note that this proof relies upon analysing fourth moments, as opposed to our proof of the $t=1$ case in \cite{EFF1}, which relied upon analysing the third moments of a non-negative function. For \(t >1\), there is no simple relationship between \(h_t\) and the non-negative function
\[v_t = \sum_{T} a_T 1_T,\]
which is the natural analogue of non-negative function we used in the \(t=1\) case. The relationship between \(\mathbb{E}[h_t^3]\) and \(\mathbb{E}[v_t^3]\) seems too hard to analyse for general \(t\), and we were unable to find another non-negative function whose third moment was simply related to that of \(h_t\). This led us to consider \(\mathbb{E}[h_t^4]\) instead; being non-negative, this is easier to bound from below than \(\mathbb{E}[h_t^3]\). The price we pay is that the relationship between \(\mathbb{E}[h_t^4]\) and \(\sum_T b_T^4\) is considerably more complicated than that between \(\mathbb{E}[h_t^3]\) and \(\sum_T b_T^3\); we make do with an approximate relationship, as opposed to the exact one in \cite{EFF1}.
\end{myproof}
We now begin our formal proof.
\subsection*{Proof of Theorem \ref{thm:main}}
Let \(\mathcal{A}\),\(f\),$\epsilon$ satisfy the hypotheses of Theorem \ref{thm:main}. If $\epsilon > 1/2$, then provided $C_t$ is chosen to be sufficiently large, the conclusion of the theorem holds trivially, for any union $\mathcal{C}$ of $\round c$ $t$-cosets of $S_n$. Moreover, for any $c_t >0$, if $c \geq c_t \sqrt{n}$, and $C_t$ is chosen to be sufficiently large, the conclusion of the theorem holds trivially. Hence, we may assume throughout that $0 \leq \epsilon \leq 1/2$ and $c \leq c_t \sqrt{n}$, where $c_t>0$ is to be chosen later.

For each $t$-coset $T$ of $S_n$, define
\[a_{T} = \frac{|\mathcal{A} \cap T|}{(n-t)!}.\]
Our ultimate aim is to show that approximately $c$ of the $a_T$'s are close to 1. Observe that
$$a_T = (n)_t \langle f,1_{T} \rangle = (n)_t \langle f_t, 1_{T} \rangle \quad \forall T \in \mathcal{C}(n,t),$$
since $1_{T} \in U_t$.

Define
$$g_t = f_t - f_{t-1};$$
note that $g_t$ is the orthogonal projection of $f$ onto $V_t$. Instead of working with the $a_T$'s, for most of the proof we will work with the quantities
\[b_{T} := (n)_t \langle g_t, 1_{T} \rangle = \mathbb{E}[g_{t}|_{T}].\]

By Lemma \ref{lemma:smallnorm}, we have
\begin{equation}\label{eq:fg}\|f_{t-1}\|_2^2 \leq C_t' n^{t-1} (c/(n)_t)^2 = O_t(c/n) c/(n)_t = O_t(c/n) \mathbb{E}[f^2],\end{equation}
so if \(c = o(n)\), \(f_t\) is close to \(g_t := f_t - f_{t-1}\) in \(L^2\)-norm, as desired.

It follows that $b_T$ is close to $a_T$, for each $t$-coset $T$. Indeed, we have
\begin{align}\label{eq:abclose}|a_T - b_T| & = (n)_t |\langle f_{t-1},1_{T} \rangle| \nonumber \\
& \leq (n)_t \|f_{t-1}\|_2 \|1_T\|_2 \nonumber \\
& \leq (n)_t  \sqrt{C_t' n^{t-1}} (c/(n)_t) \sqrt{1/(n)_t} \nonumber \\
& \leq O_t(c/\sqrt{n}),
\end{align}
using the Cauchy--Schwarz inequality and (\ref{eq:fg}).

Since \(a_T \in [0,1]\), it follows that
\begin{equation}\label{eq:boundedinterval}-O_t(c/\sqrt{n}) \leq b_T \leq 1+O_t(c/\sqrt{n})\quad \forall T\in \mathcal{C}(n,t),\end{equation}
so $b_T$ cannot lie too far outside the interval $[0,1]$.

We now define the function
\[h_{t} = \sum_{T \in \mathcal{C}(n,t)} b_T 1_{T} = (n)_t \sum_{T \in \mathcal{C}(n,t)} \langle g_t,1_T \rangle 1_{T};\]
this function is the crucial one in our proof. Our next step is to show that \(h_t\) must be close to \(g_t\). To this end, we define the linear operator

\begin{align*}
M\colon \mathbb{C}[S_n] & \to \mathbb{C}[S_n];\\
g & \mapsto (n)_t\sum_{T \in \mathcal{C}(n,t)} \langle g, 1_{T} \rangle 1_T.
\end{align*}

\noindent (Note that \(h_t = Mg_t\).) We now prove the following.

\begin{lemma}
\label{lemma:M}
The eigenspaces of \(M\) are precisely the \(U_{\alpha}\). The operator $M|_{V_t}$ is an invertible endomorphism of $V_t$, and all its eigenvalues are \(1+O_t(1/n)\).
\end{lemma}

\begin{proof}
For each \(\sigma \in S_n\), let us write \(e_{\sigma}\) for the function on \(S_n\) which is \(1\) at \(\sigma\) and \(0\) elsewhere. With slight abuse of notation, we use \((M_{\sigma,\pi})_{\sigma,\pi \in S_n}\) to denote the matrix of \(M\) with respect to the standard basis \(\{e_{\sigma}\colon \sigma \in S_n\}\) of \(\mathbb{C}[S_n]\). Observe that
\begin{align*} M_{\sigma,\pi} & = (M e_{\pi})(\sigma)\\
 & = (n)_t \sum_{T\in \mathcal{C}(n,t)} \tfrac{1}{n!} 1\{\pi \in T\}1\{\sigma \in T\}\\
& = \frac{(n)_t}{n!} |\{T \in \mathcal{C}(n,t) : \sigma,\pi \in T\}|\\
& = \frac{(n)_t}{n!} |\{x \in [n]^{(t)} : \sigma(i) = \pi(i)\ \forall i \in x\}|\\
& = \frac{(n)_t}{n! t!} |\{(i_1,\ldots,i_t) \in [n]^t : i_1,\ldots,i_t\ \textrm{are all distinct},\ \sigma(i_k) = \pi(i_k)\ \forall k \in [t]\}|\\
& = \tfrac{1}{n!} \binom{n}{t} \xi_{(n-t,1^t)} (\sigma \pi^{-1}).\end{align*}
(Here and in the rest of the paper, $1\{B\}$ is equal to $1$ if $B$ is true, and to $0$ if $B$ is false.)

Define
\[w_{\sigma} = \tfrac{1}{n!} \binom{n}{t} \xi_{(n-t,1^t)}(\sigma);\]
observe that \(\sigma \mapsto w_{\sigma}\) is a class function, and that \(M_{\sigma,\pi} = w(\sigma \pi^{-1})\) for all \(\sigma,\pi \in S_n\). By (\ref{eq:wnormalcayley}), the eigenvalues of \(M\) are given by
\begin{align}\label{eq:evaleq} \lambda_{\alpha} & = \frac{\binom{n}{t}}{n!f^{\alpha}} \sum_{\sigma \in S_n} \xi_{(n-t,1^t)}(\sigma) \chi_{\alpha}(\sigma) \nonumber \\
& = \frac{\binom{n}{t}}{f^\alpha}\langle \xi_{(n-t,1^t)},\chi_{\alpha}\rangle \nonumber \\
& = \frac{\binom{n}{t}}{f^{\alpha}} K_{\alpha,(n-t,1^t)}\quad (\alpha \vdash n),\end{align}
with the \(U_{\alpha}\) being the corresponding eigenspaces.

In particular, it follows that if \(g \in V_t = \bigoplus_{\alpha \vdash n\colon \alpha_1 = n-t} U_{\alpha}\), then \(Mg \in V_t\) also, so $M|_{V_t}$ is a linear endomorphism of $V_t$. Since $K_{\alpha,(n-t,1^t)} \geq 1$ for all $\alpha \geq (n-t,1^t)$, we have $\lambda_\alpha > 0$ for each $\alpha \vdash n$ with $\alpha_1=n-t$, so $M|_{V_t}$ is invertible. Combining Lemma \ref{lemma:kostkaestimate} and (\ref{eq:evaleq}), we see that for each \(\alpha \vdash n\) with \(\alpha_1 = n-t\), we have \(\lambda_{\alpha} = 1+O_t(1/n)\), completing the proof.
\end{proof}

Since \(g_t \in V_t\), it follows that \(h_t = Mg_t \in V_t\). Moreover,
\begin{equation}\label{eq:hg} \|h_t - g_t\|_2 = \|Mg_t - g_t\|_2 \leq O_t(1/n) \|g_t\|_2 < O_t(1/n) \sqrt{c/(n)_t}.\end{equation}
Combining (\ref{eq:fg}) and (\ref{eq:hg}), using the triangle inequality, yields
\begin{align}
\label{eq:close}
\|h_t - f\|_2 & \leq \|h_t - g_t\|_2 + \|g_t - f_t\|_2 + \|f_t - f\|_2 \nonumber\\
& \leq O_t(1/n)\sqrt{c/(n)_t} + O_t(\sqrt{c/n}) \sqrt{c/(n)_t} + \sqrt{\epsilon}\sqrt{c/(n)_t} \nonumber\\
& = (O_t(1/n)+O_t(\sqrt{c/n})+\sqrt{\epsilon}) \sqrt{c/(n)_t}\nonumber \\
& = \psi \sqrt{c/(n)_t} \nonumber \\
& = \psi \|f\|_2,\end{align}
where \(\psi := \sqrt{\epsilon}+O_t(1/n)+O_t(\sqrt{c/n})\). 

Moreover, by (\ref{eq:hg}), we have
\begin{equation} \label{eq:eh2upper} |\|h_t\|_2 - \|g_t\|_2| \leq \|h_t - g_t\|_2 \leq O_t(1/n)\|g_t\|_2 < O_t(1/n) \sqrt{c/(n)_t},\end{equation}
and therefore
\begin{equation}\label{eq:eh2upperbound} \mathbb{E}[h_t^2] \leq (1+O_t(1/n)) c/(n)_t.\end{equation}

By (\ref{eq:close}), we have
\[\|h_t\|_2 \geq (1-\psi)\|f\|_2,\]
and therefore
\begin{equation} \label{eq:eh2} \mathbb{E}[h_t^2] \geq (1-\psi)^2 c/(n)_t \geq (1-2\psi)c/(n)_t,\end{equation}
so \(\mathbb{E}[h_t^2]\) is close to \(c/(n)_t\).

We can simplify our expression for $\psi$ by noting that \(c\) cannot be too small. Indeed, since we are assuming that $\epsilon \leq 1/2$, we have
\[ \|f_t\|_2^2 \geq (\|f\|_2 - \|f-f_t\|_2)^2 \geq (1-\sqrt{\epsilon})^2 c/(n)_t \geq (1-\sqrt{1/2})^2 c/(n)_t, \]
whereas by Lemma \ref{lemma:smallnorm}, we have
\[\|f_t\|_2^2 \leq C_{t+1}' n^t (c/(n)_t)^2.\]
It follows that
\begin{equation} \label{eq:c-lb} c \geq (1-\sqrt{1/2})^2/(C_{t+1}). \end{equation}
Therefore, we can absorb the $O_t(1/n)$ term in the $O_t(\sqrt{c/n})$ term in our expression for $\psi$, giving $\psi = \sqrt{\epsilon}+ O_t(\sqrt{c/n})$. Combining this with (\ref{eq:eh2}) and (\ref{eq:eh2upperbound}) implies the following.
\begin{proposition}
\label{prop:2sided}
$$ (1-2\sqrt{\epsilon}- O_t(\sqrt{c/n}))c/(n)_t  \leq \mathbb{E}[h_t^2] \leq (1+O_t(1/n))c/(n)_t.$$
\end{proposition}
\begin{flushright} \qed \end{flushright}
Our next step is to obtain a lower bound on \(\mathbb{E}[h_t^4]\), using (\ref{eq:close}). To this end, we prove the following.

\begin{lemma}
\label{lemma:h4lowerbound}
Let $C>0$ and let $\eta,\theta \in [0,1]$ such that $\eta \leq C\theta$. Let \(F\colon [0,1] \to \{0,1\}\) be a Lebesgue measurable function with \(\mathbb{E}[F] = \theta\), and let \(h\colon [0,1] \to \mathbb{R}\) be a Lebesgue measurable function such that \(\mathbb{E}[(h-F)^2] \leq \eta\). Then
\[\mathbb{E}[h^4] \geq \theta - 4(1+C)\sqrt{\eta\theta}.\]
\end{lemma}
\begin{proof}
We will solve the following optimization problem.\\
\\
\textit{Problem} \(P\). Among all measurable functions $h\colon [0,1] \rightarrow \mathbb{R}$ such that $\mathbb{E}[(h-F)^2] \leq \eta$, find the minimum value of $\mathbb{E}[h^4]$.\\

Let \(A = F^{-1}(\{1\})\), and let \(B = F^{-1}(\{0\})\); then \(A\) and \(B\) are measurable sets, and \(\lambda(A) = \theta\). Observe that if \(h\colon [0,1] \to \mathbb{R}\) is feasible for \(P\), then the function
\[\tilde{h}(x) = \begin{cases} \frac{1}{\theta} \int_{A}h(x) \ud x & \textrm{if } x \in A, \\ \frac{1}{1-\theta} \int_{B}h(x) \ud x & \textrm{if } x \in B, \end{cases}\]
obtained by averaging \(h\) first over \(A\) and then over \(B\), is also feasible. Indeed, 
\begin{align*}\mathbb{E}[(h-F)^2] & = \int_{A} (h(x)-1)^2 \ud x + \int_{B} h(x)^2 \ud x \\
& \geq \frac{1}{\theta}\left(\int_{A} (h(x) - 1) \ud x \right)^2 + \frac{1}{1-\theta}\left(\int_{B} h(x) \ud x \right)^2\\
& = \theta \left(\frac{1}{\theta} \int_{A} h(x) \ud x - 1 \right)^2 + (1-\theta)\left(\frac{1}{1-\theta} \int_{B} h(x) \ud x \right)^2\\
& = \mathbb{E}[(\tilde{h}-F)^2],
\end{align*}
by the Cauchy--Schwarz inequality. Moreover, we have
\begin{align*}\mathbb{E}[h^4] & = \int_{A} h(x)^4 \ud x + \int_{B} h(x)^4 \ud x \\
& = \theta \cdot \frac{1}{\theta} \int_{A} h(x)^4 \ud x + (1-\theta) \cdot \frac{1}{1-\theta}\int_{B} h(x)^4 \ud x \\
& \geq \theta \left(\frac{1}{\theta}\int_{A} h(x) \ud x\right)^4 + (1-\theta)\left(\frac{1}{1-\theta}\int_{B} h(x) \ud x\right)^4\\
& = \mathbb{E}[\tilde{h}^4],
\end{align*}
by the convexity of \(y \mapsto y^4\). Hence, replacing \(h\) with \(\tilde{h}\) if necessary, we may assume that \(h\) is constant on \(A\) and on \(B\). In other words, we may assume that \(h\) has the following form:
\[h(x) = \begin{cases} r & \textrm{if }x \in A, \\ s & \textrm{if }x \in B. \end{cases}\]

Therefore, \(P\) is equivalent to the following problem:\\
\\
\textit{Problem} \(Q\):
\begin{align*}
\textrm{Minimize}\quad & \theta r^4 + (1-\theta)s^4\\
\textrm{subject to}\quad & \theta (r-1)^2 + (1-\theta)s^2 \leq \eta.
\end{align*}

Clearly $|r-1| \leq \sqrt{\eta/\theta}$ and so the minimum is at least $\theta (1-\sqrt{\eta/\theta})^4$. Conversely, $r = 1-\sqrt{\eta/\theta}$, $s = 0$ is a feasible solution to \(Q\), and so the minimum is obtained at
\[
\theta (1-\sqrt{\eta/\theta})^4 \geq \theta - 4 \sqrt{\eta \theta} - 4 \theta \sqrt{\eta/\theta} \cdot (\eta/\theta) \geq \theta - 4(1+C)\sqrt{\eta\theta}. \qedhere
\]
\end{proof}

We apply Lemma \ref{lemma:h4lowerbound} to functions on the discrete probability space \((S_n,\mathbb{P})\) by considering them as step functions on \([0,1]\). Applying it with \(h=h_t\), \(F=f\), \(\theta = c/(n)_t\), \(\eta = \psi^2 c/(n)_t\) and $C = \psi^2$, we obtain
\begin{equation} \label{eq:eh4} \mathbb{E}[h_t^4] \geq (1-4(1+\psi^2)\psi) c/(n)_t. \end{equation}

The next lemma will be used to relate \(\mathbb{E}[h_t^2]\) to \(\sum_{T \in \mathcal{C}(n,t)}b_T^2\).

\begin{lemma}
\label{lemma:b2}
Suppose $g \in V_t$ is real-valued, let $b_{T} = (n)_t \langle g, 1_{T} \rangle$ for each $T \in \mathcal{C}(n,t)$, and define \(h \in V_t\) by
\[h = Mg = \sum_{T \in \mathcal{C}(n,t)} b_T 1_T.\]
Then
\[\mathbb{E}[h^2] = (1+O_t(1/n)) \frac{1}{(n)_t} \sum_{T \in \mathcal{C}(n,t)} b_{T}^2.\]
\end{lemma}
\begin{proof}
Let us abbreviate $\sum_{T \in \mathcal{C}(n,t)}$ to $\sum_{T}$. We expand as follows:
\[\mathbb{E}[h^2] = \sum_{T} b_T^2 \mathbb{E}[1_T] + \sum_{(S,T)\colon S \neq T} b_S b_T \mathbb{E}[1_{S}1_{T}] = \frac{1}{(n)_t} \sum_{T}b_T^2 + \sum_{\substack{(S,T)\colon S \neq T\\ S \cap T \neq \emptyset}} b_S b_T \mathbb{E}[1_{S \cap T}].\]
We must prove that
\begin{equation} \label{eq:expansion} \sum_{\substack{(S,T)\colon S \neq T \\ S \cap T \neq \emptyset}} b_S b_T \mathbb{E}[1_{S \cap T}] = O_t(1/(n)_{t+1}) \sum_{T}b_T^2.\end{equation}

We pause to define some more notation. We define
\[\mathcal{D}(n,t) := \{\{(x_1,y_1),\ldots,(x_t,y_t)\} : (x_1,\ldots,x_t) \in ([n])_t,\ (y_1,\ldots,y_t) \in ([n])_t\};\]
note that there is a natural one-to-one correspondence between \(\mathcal{D}(n,t)\) and \(\mathcal{C}(n,t)\), the set of $t$-cosets, given by identifying the \(t\)-coset
\[\{\sigma \in S_n : \sigma(x_i)=y_i\ \forall i \in [t]\} \in \mathcal{C}(n,t)\]
with the set of ordered pairs
\[\{(x_1,y_1),\ldots,(x_t,y_t)\} \in \mathcal{D}(n,t).\]
We denote this correspondence by \(\leftrightarrow\). If \(T \leftrightarrow I\), we define \(b_I := b_T\).

We say that two sets \(\{(u_1,v_1),\ldots,(u_s,v_s)\} \in \mathcal{D}(n,s),\ \{(x_1,y_1),\ldots,(x_t,y_t)\} \in \mathcal{D}(n,t)\) are {\em compatible} if the corresponding cosets have nonempty intersection, i.e. if $u_i = x_j \Leftrightarrow v_i = y_j$. Otherwise, we say that they are {\em incompatible}.

We say that the two sets are {\em independent} if \(u_i \neq x_j\) and \(v_i \neq y_j\) for all \(i \in [s],\ j \in [t]\). Observe that if \(A,B \in \mathcal{D}(n,t)\) are compatible, then we may express
\[A = I \cup J,\ B = I \cup K,\]
where \(I \in \mathcal{D}(n,e)\) for some \(e \leq t\), \(J,K \in \mathcal{D}(n,t-e)\), and \(I,J\) and \(K\) are pairwise independent.

We pause to observe some crucial linear dependencies between the \(b_T\)'s. Firstly, we claim that
\begin{equation}\label{eq:ld1} \sum_{k \neq x_1,\ldots,x_{t-1}} b_{\{(x_1,y_1),\ldots,(x_{t-1},y_{t-1}),(k,y_t)\}} = 0\end{equation}
for any distinct \(x_1,\ldots,x_{t-1} \in [n]\), and any distinct \(y_1,\ldots,y_{t} \in [n]\). Indeed, the left-hand side is precisely \((n)_t \langle g, 1_S \rangle\), where 
\[\mathcal{C}(n,t-1) \ni S = \{\sigma \in S_n : \sigma(x_i)=y_i\ \forall i \in [t-1]\} \leftrightarrow \{(x_1,y_1),\ldots,(x_{t-1},y_{t-1})\}.\]
Since $g \in V_t \perp U_{t-1} = \textrm{Span}\{1_{S} : S \in \mathcal{C}(n,t-1)\}$, the claim follows.

Similarly, we claim that
\begin{equation} \label{eq:ld2} \sum_{k \neq y_1,\ldots,y_{t-1}} b_{\{(x_1,y_1),\ldots,(x_{t-1},y_{t-1}),(x_t,k)\}} = 0,\end{equation}
for any distinct \(x_1,\ldots,x_{t} \in [n]\), and any distinct \(y_1,\ldots,y_{t-1} \in [n]\). Indeed, the left-hand side is precisely \((n)_t \langle g, 1_S \rangle\), where $S$ is as above.

We split up the sum in (\ref{eq:expansion}) as follows:
\[\sum_{\substack{(S,T)\colon S \neq T,\\ S \cap T \neq \emptyset}} b_S b_T \mathbb{E}[1_{S \cap T}] = \sum_{e=0}^{t-1} \frac{1}{(n)_{2t-e}} \sum_{\substack{I \in \mathcal{D}(n,e),\ J,K \in \mathcal{D}(n,t-e),\\
I,J,K \textrm{ independent}}}b_{I \cup J}b_{I \cup K}.\]
We now use an argument similar to the indicator-function proof of inclusion-exclusion. Observe that for any \(I \in \mathcal{D}(n,e)\), we have
\[\sum_{\substack{J,K \in \mathcal{D}(n,t-e)\colon\\ I,J,K \textrm{ indep.}}} b_{I \cup J}b_{I \cup K} = \sum_{\substack{J,K \in \mathcal{D}(n,t-e)\colon\\ I,K \textrm{ indep.,}\ I,J \textrm{ indep.}}} b_{I \cup J}b_{I \cup K} \prod_{p,q \in [t-e],d \in [2]} (1-1\{j^{(d)}_p = k^{(d)}_q\}),\]
where
\begin{align*}J&=\{(j_1^{(1)},j_1^{(2)}),(j_2^{(1)},j_2^{(2)}),\ldots,(j_{t-e}^{(1)},j_{t-e}^{(2)})\},\\
K&=\{(k_1^{(1)},k_1^{(2)}),(k_2^{(1)},k_2^{(2)}),\ldots,(k_{t-e}^{(1)},k_{t-e}^{(2)})\}.
\end{align*}
Consider what happens when we multiply out the product on the right-hand side, {\em \`a la} inclusion-exclusion. We obtain a $(\pm 1)$-linear combination of products of indicators of the form
$$\prod_{(p,q,d) \in \mathcal{I}} 1\{j^{(d)}_p = k^{(d)}_q\}.$$
For each such (non-zero) product of indicators, let us replace $k_{q}^{(d)}$ by $j_{p}^{(d)}$ in the sum whenever the product contains the indicator $1\{j^{(d)}_p = k^{(d)}_q\}$. We obtain a $(\pm 1)$-linear combination of terms of the form
\[\sum_{\substack{J,K'\colon \\\ I,J \textrm{ indep,}\ I,K' \textrm{ indep. }}}b_{I \cup J} b_{I \cup K'},\]
where $K'$ is obtained from $K$ by replacing \(k_{q}^{(d)}\) with $j_{p}^{(d)}$ for various $(p,q,d)$. If $K'$ contains any (unreplaced) $k_{q}^{(d)}$, then the corresponding term is zero, by applying (\ref{eq:ld1}) or (\ref{eq:ld2}) with $k=k_{q}^{(d)}$. Hence, the only non-zero terms are where all the $k_{q}^{(d)}$'s have been replaced, i.e. they are of the form
\begin{equation}\label{eq:form}\sum_{\substack{J \in \mathcal{D}(n,t-e)\colon\\ I,J \textrm{ indep.}}} b_{I \cup J}b_{I \cup \{(j^{(1)}_1,j^{(2)}_{\sigma(1)}),\ldots,(j^{(1)}_{t-e},j^{(2)}_{\sigma(t-e)})\}},\end{equation}
where \(\sigma \in S_{t-e}\). We now sum each such term over \(I\), and apply the Cauchy--Schwarz inequality:
\[\left|\sum_{\substack{I \in \mathcal{D}(n,e),\ J \in \mathcal{D}(n,t-e)\colon\\ I,J \textrm{ indep.}}} b_{I \cup J}b_{I \cup \{(j^{(1)}_1,j^{(2)}_{\sigma(1)}),\ldots,(j^{(1)}_{t-e},j^{(2)}_{\sigma(t-e)})\}}\right| \leq \binom{t}{e}\sum_T b_T^2.\]
Here, the factor of \(\binom{t}{e}\) comes from the fact that there are \(\binom{t}{e}\) ways of producing an ordered partition \((I,J)\) of a \(t\)-set into an \(e\)-set $I$ and a \((t-e)\)-set $J$.

There are at most \(O_t(1)\) terms of the form (\ref{eq:form}) for each \(e \in [t-1]\), so we obtain
\[\left|\sum_{e=0}^{t-1} \frac{1}{(n)_{2t-e}} \sum_{\substack{I \in \mathcal{D}(n,e),\ J,K \in \mathcal{D}(n,t-e),\\
I,J,K \textrm{ independent}}}b_{I \cup J}b_{I \cup K}\right| \leq O_t(1) \frac{1}{(n)_{t+1}} \sum_{T \in \mathcal{C}(n,t)}b_T^2,\] 
as required.
\end{proof}

The next lemma will be used to relate \(\mathbb{E}[h_t^4]\) to \(\sum_{T \in \mathcal{C}(n,t)} b_T^4\).

\begin{lemma}
\label{lemma:b4}
Suppose $g \in V_t$ is real-valued, let $b_{T} = (n)_t \langle g, 1_{T} \rangle$ for each $T \in \mathcal{C}(n,t)$, and define \(h \in V_t\) by
\[h = Mg = \sum_{T \in \mathcal{C}(n,t)} b_T 1_T.\]
Then
\[\mathbb{E}[h^4] = \frac{1}{(n)_t} \left(\sum_{T \in \mathcal{C}(n,t)} b_{T}^4 + O_t(1/n) \left(\sum_{T\in \mathcal{C}(n,t)} b_T^2\right)^2\right).\]
\end{lemma}
\begin{proof}
We need one more piece of notation. If $X$ is a finite set, and \(Z \subset X^{2}\) is a set of ordered pairs of elements of $X$, we let \(Z(d)\) denote the set of all elements of $X$ which occur in the \(d\)th coordinate of an ordered pair in \(Z\), for $d \in \{1,2\}$.

We expand as follows:
\begin{align*} \mathbb{E}[h^4] & = \sum_{T} b_T^4 \mathbb{E}[1_T] + \sum_{(T_1,T_2,T_3,T_4) \textrm{ not all equal}} b_{T_1} b_{T_2} b_{T_3} b_{T_4} \mathbb{E}[1_{T_1}1_{T_2}1_{T_3}1_{T_4}]\\
& = \frac{1}{(n)_t} \sum_{T}b_T^4 + \sum_{\substack{(T_1,T_2,T_3,T_4) \textrm{ not all equal}\colon\\\ T_1 \cap T_2 \cap T_3 \cap T_4 \neq \emptyset}} b_{T_1} b_{T_2} b_{T_3} b_{T_4} \mathbb{E}[1_{T_1 \cap T_2 \cap T_3 \cap T_4}].\end{align*}
We must prove that
\[\sum_{\substack{(T_1,T_2,T_3,T_4) \textrm{ not all equal}\colon\\ T_1 \cap T_2 \cap T_3 \cap T_4 \neq \emptyset}} b_{T_1} b_{T_2} b_{T_3} b_{T_4} \mathbb{E}[1_{T_1 \cap T_2 \cap T_3 \cap T_4}] = O_t(1/(n)_{t+1}) \left(\sum_T b_T^2\right)^2.\]
We now split up the sum above, by partitioning the set
\[\{(T_1,T_2,T_3,T_4) \textrm{ not all equal} : T_1 \cap T_2 \cap T_3 \cap T_4 \neq \emptyset\}\]
into a bounded number of classes. We define an equivalence relation $\sim$ on this set by
$$(T_1',T_2',T_3',T_4') \sim (T_1,T_2,T_3,T_4)\quad \textrm{iff there exist } \sigma,\pi \in S_n :  \sigma T_r' \pi = T_r\ \forall r \in [4].$$
Note that the number of equivalence classes is at most $\binom{4t}{t}^4 = O_t(1)$, since we can choose a representative from each equivalence class which is a 4-tuple of $t$-cosets, each of which is the pointwise stabiliser of some $t$-element subset of $[4t]$.

Note also that \(\mathbb{E}[1_{T_1 \cap T_2 \cap T_3 \cap T_4}]\) depends only upon the equivalence class of \((T_1,T_2,T_3,T_4)\), and that we always have \(\mathbb{E}[1_{T_1 \cap T_2 \cap T_3 \cap T_4}] = 1/(n)_{l}\), where \(t+1 \leq l \leq 4t\), and therefore we always have \(\mathbb{E}[1_{T_1 \cap T_2 \cap T_3 \cap T_4}] \leq 1/(n)_{t+1}\).

Let \(\mathcal{S}\) be an equivalence class; we wish to bound
\[\left|\sum_{(T_1,T_2,T_3,T_4) \in \mathcal{S}} b_{T_1} b_{T_2} b_{T_3} b_{T_4} \mathbb{E}[1_{T_1 \cap T_2 \cap T_3 \cap T_4}]\right|.\]
Since we always have \(\mathbb{E}[1_{T_1 \cap T_2 \cap T_3 \cap T_4}] \leq 1/(n)_{t+1}\), we have
\begin{equation}\label{eq:uniformb}\left|\sum_{(T_1,T_2,T_3,T_4) \in \mathcal{S}} b_{T_1} b_{T_2} b_{T_3} b_{T_4} \mathbb{E}[1_{T_1 \cap T_2 \cap T_3 \cap T_4}]\right| \leq \frac{1}{(n)_{t+1}} \left| \sum_{(T_1,T_2,T_3,T_4) \in \mathcal{S}} b_{T_1} b_{T_2} b_{T_3} b_{T_4}\right|.\end{equation}

Take any \((T_1, T_2, T_3, T_4) \in \mathcal{S}\), and let \(T_r \leftrightarrow Z_r \in \mathcal{D}(n,t)\) for each \(r \in [4]\). Since \(T_1 \cap T_2 \cap T_3 \cap T_4 \neq \emptyset\), \(Z_1,Z_2,Z_3,Z_4\) are pairwise compatible, so for any $r,s \in [4]$, if $Z_r = \{(u_i,v_i)\}_{i \in [t]}$ and $Z_s = \{(x_i,y_i)\}_{i \in [t]}$, then \(u_i = x_j\ \Leftrightarrow v_i = y_j\) for all \(i,j \in [t]\). It follows that \(|\cup_{r=1}^{4}Z_r(1)| = |\cup_{r=1}^{4}Z_r(2)|\); we denote this number by \(N\). Note that \(N \leq 4t\).

Observe that we may write
\begin{equation} \label{eq:bad} \sum_{(T_1,T_2,T_3,T_4) \in \mathcal{S}} b_{T_1} b_{T_2} b_{T_3} b_{T_4} = \sum_{\substack{(i_1,\ldots,i_N) \textrm{ distinct},\\ (i_{N+1},\ldots,i_{2N}) \textrm{ distinct}}} \prod_{r=1}^{4} b_{\{(i_p,i_q)\}_{(p,q) \in A_r}}\end{equation}
where \(A_r \subset [N] \times \{N+1,\ldots,2N\}\) with \(|A_r|=|A_r(1)|=|A_r(2)|=t\), for each \(r \in [4]\). (The conditions $|A_r(1)|=|A_r(2)|=t$ say that any $p \in [2N]$ occurs in at most one ordered pair in $A_r$.)

For each \(p \in [2N]\), we shall say that \(p\) is {\em `good'} if it appears in at least two of the \(A_r\)'s; otherwise, we say that \(p\) is {\em `bad'}. Our aim is to use indicator functions ({\em \`a la} inclusion-exclusion), together with the linear dependence relations (\ref{eq:ld1}) and (\ref{eq:ld2}), to rewrite (\ref{eq:bad}) as a \((\pm 1)\)-linear combination of at most \(O_t(1)\) sums containing no `bad' indices. Without loss of generality, we may assume that there is a bad index in $\{N+1,\ldots,2N\}$, say $2N$. (The argument for a bad index in \([N]\) is identical, except that (\ref{eq:ld1}) is used instead of (\ref{eq:ld2}).) Without loss of generality, we may assume that $2N$ appears in \(A_4\) alone. Let
\[P = \{p \in \{N+1,\ldots,2N\} : p \textrm{ appears in }A_4\},\quad Q = \{N+1,\ldots,2N\} \setminus P.\]
Then
\begin{align*}& \sum_{\substack{(i_1,\ldots,i_N) \textrm{ distinct},\\ (i_{N+1},\ldots,i_{2N}) \textrm{ distinct}}} \prod_{r=1}^{4} b_{\{(i_p,i_q)\}_{(p,q) \in A_r}} \\
& = \sum_{\substack{(i_1,\ldots,i_N) \textrm{ distinct},\\ (i_{N+1},\ldots,i_{2N-1}) \textrm{ distinct},\\ i_{2N}\colon i_{2N} \neq i_{p}\ \forall p \in P}} \prod_{r=1}^{4} b_{\{(i_p,i_q)\}_{(p,q) \in A_r}}\prod_{q \in Q}(1-1\{i_{2N} = i_{q}\}).\end{align*}

Consider what happens when we expand out the product over all \(q \in Q\). We obtain
\begin{align*} & \sum_{\substack{(i_1,\ldots,i_N) \textrm{ distinct},\\ (i_{N+1},\ldots,i_{2N-1}) \textrm{ distinct},\\ i_{2N}\colon i_{2N} \neq i_{p}\ \forall p \in P}} \prod_{r=1}^{4} b_{\{(i_p,i_q)\}_{(p,q) \in A_r}} \\
- & \sum_{q \in Q} \sum_{\substack{(i_1,\ldots,i_N) \textrm{ distinct},\\ (i_{N+1},\ldots,i_{2N-1}) \textrm{ distinct}}} b_{\{(i_p,i_q)\}_{(p,q) \in A_4(2N \rightarrow q)}} \prod_{r=1}^{3} b_{\{(i_p,i_q)\}_{(p,q) \in A_r}},
\end{align*}
where \(A_4(2N \rightarrow q)\) is produced from \(A_4\) by replacing \(2N\) with \(q\). We have
\[\sum_{\substack{(i_1,\ldots,i_N) \textrm{ distinct},\\ (i_{N+1},\ldots,i_{2N-1}) \textrm{ distinct},\\ i_{2N}\colon i_{2N} \neq i_{p}\ \forall p \in P}} \prod_{r=1}^{4} b_{\{(i_p,i_q)\}_{(p,q) \in A_r}} = 0,\]
by (\ref{eq:ld2}), so the first term above is zero; none of the other terms involve the `bad' index \(2N\). Hence, we have
\begin{align*} & \sum_{\substack{(i_1,\ldots,i_N) \textrm{ distinct},\\\ (i_{N+1},\ldots,i_{2N}) \textrm{ distinct}}} \prod_{r=1}^{4} b_{\{(i_p,i_q)\}_{(p,q) \in A_r}} \\
& = - \sum_{q \in Q} \sum_{\substack{(i_1,\ldots,i_N) \textrm{ distinct},\\ (i_{N+1},\ldots,i_{2N-1}) \textrm{ distinct}}} b_{\{(i_p,i_q)\}_{(p,q) \in A_4(2N \rightarrow q)}} \prod_{r=1}^{3} b_{\{(i_p,i_q)\}_{(p,q) \in A_r}}.
\end{align*}

We now have \(|Q|\) sums to deal with, but each has one less bad index than the original sum. By repeating this process until there are no more bad indices remaining in any sum, we can express the original sum (\ref{eq:bad}) as a (\(\pm 1\))-linear combination of at most \(O_t(1)\) terms of the form:
\begin{equation}\label{eq:crudeb}\sum_{(i_1,\ldots,i_{K}) \in \mathcal{R}} \prod_{r=1}^{4} b_{\{(i_p,i_q)\}_{(p,q) \in \tilde{A}_r}},\end{equation}
where
\begin{itemize}
\item \(K \leq 2N \leq 8t\),
\item \(\mathcal{R} \subset [n]^{K}\) is a subset defined by constraints of the form \(i_k \neq i_l\),
\item For each $r \in [4]$, we have $\tilde{A}_r \subset [K]^2$, and each $k \in [K]$ occurs in at most one ordered pair in $\tilde{A_r}$,
\item Each index \(k \in [K]\) is `good', meaning that it appears in at least two of the \(\tilde{A}_r\)'s.
\end{itemize}
Note that the $\tilde{A}_r$'s vary with the term we are looking at. Crudely, (\ref{eq:crudeb}) is at most
\begin{equation}\label{eq:est} \sum_{(i_1,\ldots,i_{K}) \in [n]^{K}} \prod_{r=1}^{4} |b_{\{(i_p,i_q)\}_{(p,q) \in \tilde{A}_r}}|,\end{equation}
where we define \(b_{\{(i_k,j_k)\}_{k \in [t]}} = 0\) if \(\{(i_k,j_k)\}_{k \in [t]} \notin \mathcal{D}(n,t)\), i.e. if \(i_k = i_{l}\) or \(j_k = j_l\) for some \(k \neq l\).
To bound (\ref{eq:est}), we need the following claim.
\begin{claim}
\label{claim:cs}
Let \(X\) be a finite set, let \(m \geq 2\), and let \(a_1,\ldots,a_m,L \in \mathbb{N}\). Suppose \(f_j\colon X^{a_j} \to \mathbb{R}\) for \(j = 1,2,\ldots,m\). Suppose \(\sigma_j\colon [a_j] \to [L]\) are injections, such that each \(\ell \in [L]\) lies in the image of at least two of these injections. Then
\[\left(\sum_{i_1,\ldots,i_L \in X} \prod_{j=1}^{m} f_j(i_{\sigma_j(1)},\ldots,i_{\sigma_j(a_j)})\right)^2 \leq \prod_{j=1}^{m} \left(\sum_{i_1,\ldots,i_{a_j} \in X} f_j(i_1,\ldots,i_{a_j})^2\right).\]
\end{claim}
\begin{proof}
This follows immediately from \cite[Lemma 3.2]{FriedgutEntropy}, which is a weighted version of Shearer's entropy lemma. (We apply the latter with $r = m$ and $t=2$ to the hypergraph $(V,E)$ where $V = [L] \times X$ and $E = \{\{(1,x_1),(2,x_2),\ldots,(L,x_L)\}\ :\ x_1,\ldots,x_L \in X\}$, with $F_j = \sigma_j([a_j]) \times X$ for each $j \in [m]$ and with
$$w_j(\{(l,x_l):\ l \in \sigma_j([a_j])\}) = f_j(x_{\sigma_j(1)},\ldots,x_{\sigma_j(a_j)})$$
for each $j \in [m]$.)
\end{proof}
We shall apply Claim \ref{claim:cs} to the function $F\colon [n]^{2t}\to \mathbb{R}$ defined by
$$F(i_{1},\ldots,i_{t},j_{1},\ldots,j_{t}) = \left\{\begin{array}{ll} b_{\{(i_k,j_k)\}_{k \in [t]}} & \textrm{ if }i_1,\ldots,i_t \textrm{ are all distinct}\\
& \textrm{ and }j_1,\ldots,j_t \textrm{ are all distinct},\\
0 & \textrm{ otherwise.}\end{array}\right.$$
 For each $r \in [4]$, choose any ordering $((p_{1},q_{1}),\ldots,(p_t,q_t))$ of the ordered pairs in $\tilde{A}_r$, and define the injection $\sigma_r\colon [2t]\to [K]$ by
 $$\sigma_r(w) = \left\{\begin{array}{ll} p_w & \textrm{ if }1 \leq w \leq t,\\ q_{w-t} & \textrm{ if }t+1 \leq w \leq 2t.\end{array}\right.$$
Applying Claim \ref{claim:cs} with \(m=4\), \(L=K\), \(X = [n]\), \(a_j = 2t\) for each \(j \in [4]\), $f_j=F$ for each $j \in [4]$, and $\sigma_1,\sigma_2,\sigma_3,\sigma_4$ as above, yields
\[\left| \sum_{(i_1,\ldots,i_{K}) \in [n]^{K}} \prod_{r=1}^{4} |b_{\{(i_p,i_q)\}_{(p,q) \in A_r}}|\right| \leq \left(t!\sum_T b_T^2\right)^2.\]
Here, the factor of $t!$ comes from the fact that we are summing over all possible orderings of the ordered pairs in each $Z \in \mathcal{D}(n,t)$. Since there are only $O_t(1)$ terms of the form (\ref{eq:crudeb}), we have
$$\sum_{(T_1,T_2,T_3,T_4) \in \mathcal{S}} b_{T_1} b_{T_2} b_{T_3} b_{T_4} \leq O_t(1)\left(\sum_T b_T^2\right)^2.$$
Using (\ref{eq:uniformb}), together with the fact that there are only $O_t(1)$ different equivalent classes, yields 
\[\sum_{\substack{(T_1,T_2,T_3,T_4) \textrm{ not all equal}\colon\\ T_1 \cap T_2 \cap T_3 \cap T_4 \neq \emptyset}} b_{T_1} b_{T_2} b_{T_3} b_{T_4} \mathbb{E}[1_{T_1 \cap T_2 \cap T_3 \cap T_4}] \leq O_t(1/(n)_{t+1}) \left(\sum_T b_T^2\right)^2,\]
completing the proof of Lemma \ref{lemma:b4}.
\end{proof}

Combining Proposition \ref{prop:2sided} and Lemma \ref{lemma:b2} shows that \(\sum_T b_T^2\) is close to \(c\):

\begin{proposition}
\label{prop:2sidedbound}
$$(1-2\psi - O_t(1/n))c \leq \sum_{T} b_T^2 \leq (1+O_t(1/n))c,$$
where $\psi = \sqrt{\epsilon} + O_t(\sqrt{c/n})$.
\end{proposition}
\begin{flushright} \qed \end{flushright}

Similarly, combining Lemma \ref{lemma:b4} and (\ref{eq:eh4}) yields 
\[\sum_{T} b_{T}^4 + O_t(1/n) \left(\sum_T b_T^2\right)^2 \geq (1-4(1+\psi^2)\psi)c.\]

Using the fact that \(\sum_T b_T^2 = O_t(c)\) yields
\begin{equation} \label{eq:b4new} \sum_{T} b_{T}^4 \geq (1-4(1+\psi^2)\psi)c - O_t(c^2/n) = (1-4(1+\psi^2)\psi - O_t(c/n))c = (1-\psi')c,\end{equation}
where $\psi' := 4(1+\psi^2)\psi + O_t(c/n)$. Since $\epsilon \leq 1/2$ and $c = O_t(\sqrt{n})$, we have $\psi = O_t(1)$, and so
\[ \psi' = O_t(\psi + c/n) = O_t(\sqrt{\epsilon} + \sqrt{c/n}). \]

%\begin{comment}
%We pause to observe that this implies that \(c\) cannot be too small. Indeed, suppose that \(c \leq 1\). By the convexity of \(y \mapsto y^2\), we have
%\begin{align*}
%\sum_T b_T^4 & \leq \left(\sum_T b_T^2\right)^2\\
%& \leq (1+O_t(1/n))^2c^2\\
%& \leq (1+O_t(1/n))c^2;
%\end{align*}
%combining this with (\ref{eq:b4new}) yields
%\[(1-8\sqrt{\epsilon} - O_t(1/\sqrt{n}))c < (1+O_t(1/n))c^2,\]
%and therefore
%\[c \geq 1-8\sqrt{\epsilon} - O_t(1/\sqrt{n}).\]
%
%Note that this enables us to absorb the \(O_t(1/n)\) term in the \(O_t(\sqrt{c/n})\) term, in our expressions for \(\psi\) and \(\psi'\): we have \(\psi,\psi' = \sqrt{\epsilon}+O_t(\sqrt{c/n})\).
%\end{comment}

By (\ref{eq:boundedinterval}) and Proposition \ref{prop:2sidedbound}, we have
\begin{align*} \sum_T b_T^4 & \leq (1+O_t(c/\sqrt{n}))\sum_T b_T^2\\
&  \leq (1+O_t(c/\sqrt{n}))(1+O_t(1/n))c \\
&= (1+O_t(c/\sqrt{n}))c,\end{align*}
using the fact, from \eqref{eq:c-lb}, that $c = \Omega_t(1)$. Combining this with (\ref{eq:b4new}) shows that \(\sum_T b_T^4\) is close to \(c\):

\begin{proposition}
\label{prop:b42sided}
$$(1-\psi')c \leq \sum_T b_T^4 \leq (1+O_t(c/\sqrt{n}))c,$$
where $\psi'  = O_t(\sqrt{\epsilon} + \sqrt{c/n}).$
\end{proposition}
\begin{flushright} \qed \end{flushright}

Let \(x_1,\ldots,x_N\) denote the entries \((b_{T}^2)_{T \in \mathcal{C}(n,t)}\) in non-increasing order. By Proposition \ref{prop:2sidedbound}, we have
\begin{equation}\label{eq:sumx2simpler}\sum_{k=1}^{N}x_k \leq (1+O_t(1/n))c,\end{equation}
and by Proposition \ref{prop:b42sided}, we have
\begin{equation}\label{eq:sumx4simpler}(1-\psi')c \leq \sum_{k=1}^{N}x_k^2 \leq (1+O_t(c/\sqrt{n}))c.\end{equation}

Subtracting \eqref{eq:sumx4simpler} from \eqref{eq:sumx2simpler} yields:
\begin{equation} \label{eq:skewness}
\sum_{k=1}^N x_k (1 - x_k) \leq (\psi'+O_t(1/n))c = \psi''c,
\end{equation}
where \(\psi'' := \psi'+O_t(1/n) = O_t(\sqrt{\epsilon} + \sqrt{c/n})\). By an appropriate choice of the implicit constant in the definition of $\psi''$, we may ensure that
$$\sum_{k=1}^{N} x_k \geq (1-\psi'')c,$$
by Proposition \ref{prop:2sidedbound}.

Let $m$ be the largest integer \(k\) such that $x_k \geq 1/2$ (recall that the $x_k$ are arranged in non-increasing order). Then
\[ \sum_{k=m+1}^N x_k \leq 2 \sum_{k=m+1}^N x_k (1-x_k) \leq 2\psi''c,\]
by (\ref{eq:skewness}). Therefore,
\begin{equation} \label{eq:bigterms-lb}
(1-3\psi'')c \leq \sum_{k=1}^m x_k \leq (1+O_t(c/\sqrt{n})) m.
\end{equation}
On the other hand, we have
\[ \sum_{k=1}^m (1 - x_k) \leq 2 \sum_{k=1}^m x_k (1 - x_k) \leq 2\psi''c.\]
Rearranging, we have
\begin{equation} \label{eq:bigterms}
\sum_{k=1}^m x_k \geq m - 2\psi''c.
\end{equation}
Since $2x_k - 1 \leq x_k^2$, we have
\begin{equation} \label{eq:bigterms-ub}
(1+O_t(c/\sqrt{n}))c \geq \sum_{k=1}^m x_k^2 \geq 2\sum_{k=1}^m x_k - m \geq m - 4\psi''c.
\end{equation}
Equation~\eqref{eq:bigterms-ub} shows that $m \leq (1 + O_t(\sqrt{\epsilon} + c/\sqrt{n}))c$. Since $\epsilon \leq 1/2$ and $c = O_t(\sqrt{n})$, this implies that $m = O_t(c)$.
Combining (\ref{eq:bigterms-lb}) and (\ref{eq:bigterms-ub}) now yields
\begin{equation} \label{eq:bigterms-int}
|m-c| \leq O_t(\epsilon^{1/2} + c/\sqrt{n})c.
\end{equation}

Let \(T_1,\ldots,T_m\) be the \(t\)-cosets corresponding to \(x_1,\ldots,x_m\), and let 
\[\mathcal{C}' = \bigcup_{k=1}^{m} T_{k}\]
denote the corresponding union of \(m\) \(t\)-cosets of \(S_n\).
Provided we choose $c_t$ to be sufficiently small depending on $t$, our assumption $c \leq c_t \sqrt{n}$ implies, via ~\eqref{eq:boundedinterval}, that $b_T > -1/\sqrt{2}$ for all $t$-cosets $T$. Hence, $b_T \geq 0$ whenever $b_T^2 \geq 1/2$.
We have
\begin{align*}
\sum_{k=1}^{m} |\mathcal{A} \cap T_{k}|/(n-t)! & = \sum_{k=1}^{m} a_{T_k}\\
& \geq \sum_{k=1}^{m} b_{T_k} - mO_{t}(c/\sqrt{n})\\
& \geq \frac{1}{1+O_t(c/\sqrt{n})}\sum_{k=1}^{m} b_{T_k}^2 - O_t(c^2/\sqrt{n})\\
& \geq (1 - 3\psi'' - O_t(c/\sqrt{n}))c\\
& \geq (1 - O_t(\epsilon^{1/2} + c/\sqrt{n}))c,
\end{align*}
using (\ref{eq:abclose}), (\ref{eq:boundedinterval}) and (\ref{eq:bigterms-lb}). Since $|T_{i} \cap T_{j}| \leq (n-t-1)!$ for each \(i \neq j\), we have
\[|\mathcal{A} \cap \mathcal{C}'| \geq \sum_{k=1}^m |\mathcal{A} \cap T_{k}| - \binom{m}{2} (n-t-1)! \geq (1-O_t(\epsilon^{1/2} + c/\sqrt{n}))c(n-t)!,\]
i.e. \(\mathcal{A}\) contains almost all of \(\mathcal{C}'\). Equation~\eqref{eq:bigterms-int} shows that
\begin{align*}
|\mathcal{C}'| &\leq m (n-t)! \leq (1 + O_t(\epsilon^{1/2} + c/\sqrt{n}))c (n-t)!.
\end{align*}
Since \(|\mathcal{A}| = c(n-t)!\), we must have
\[|\mathcal{A} \triangle \mathcal{C}'| = |\mathcal{A}| + |\mathcal{C}'| - 2|\mathcal{A} \cap \mathcal{C}'| = O_t( \epsilon^{1/2} + c/\sqrt{n})c(n-t)!.\]
Crudely, we have
$$|m-\round c| \leq 2|m-c| = O_t(\epsilon^{1/2} + c/\sqrt{n})c.$$
By adding or deleting $|m-\round c|$ $t$-cosets to or from $\mathcal{C}'$, we may produce a family $\mathcal{C} \subset S_n$ which is a union of $\round c$ $t$-cosets, and satisfies
\begin{align*} |\mathcal{A} \triangle \mathcal{C}| & \leq |\mathcal{A} \triangle \mathcal{C}'|+|\mathcal{C}' \triangle \mathcal{C}|\\
& = O_t(\epsilon^{1/2} + c/\sqrt{n})c(n-t)!.\end{align*}
Since \(|c - \round c| \leq |m - c| = O_t(\epsilon^{1/2} + c/\sqrt{n})c\), this completes the proof of Theorem \ref{thm:main}.

\section{An isoperimetric inequality for the transposition graph}
\label{section:isoperimetric}
In this section, we will apply Theorem \ref{thm:main} to obtain an isoperimetric inequality for \(S_n\). We first give some background and notation on discrete isoperimetric inequalities.

Isoperimetric problems are of ancient interest in mathematics. In general, they ask for the smallest possible size of the `boundary' of a set of a given `size'. Discrete isoperimetric inequalities deal with discrete notions of boundary in graphs. There are two different notions of boundary in graphs, the vertex-boundary and the edge-boundary; here, we deal with the latter.

If \(G = (V,E)\) is any graph, and \(S,T \subset V\), we write \(E_{G}(S,T)\) for the set of edges of \(G\) between \(S\) and \(T\), and we write \(e_{G}(S,T) = |E_{G}(S,T)|\). We write \(\partial_{G}S = E_{G}(S,S^c)\) for the set of edges of \(G\) between \(S\) and $S^c:=V\setminus S$; this is called the {\em edge-boundary of \(S\) in \(G\)}. An {\em edge-isoperimetric inequality for \(G\)} gives a lower bound on the minimum size of the edge-boundary of a set of size \(k\), for each integer \(k\).

The {\em transposition graph} \(T_n\) is the Cayley graph on \(S_n\) generated by the transpositions in \(S_n\); equivalently, two permutations are joined if, as sequences, one can be obtained from the other by transposing two elements. In this section, we are concerned with the edge-isoperimetric problem for \(T_n\).

It would be of great interest to prove an isoperimetric inequality for the transposition graph which is sharp for all set-sizes. Ben Efraim \cite{benefraim} conjectures that initial segments of the lexicographic order on \(S_n\) have the smallest edge-boundary of all sets of the same size. (The lexicographic order on \(S_n\) is defined as follows: if \(\sigma,\pi \in S_n\), we say that \(\sigma < \pi\) if \(\sigma(j) < \pi(j)\), where \(j = \min\{i \in [n] : \sigma(i) \neq \pi(i)\}\). The {\em initial segment of size \(k\) of the lexicographic order on \(S_n\)} simply means the smallest \(k\) elements of \(S_n\) in the lexicographic order.)

\begin{conjecture}[Ben Efraim]
\label{conj:benefraim}
For any \(\mathcal{A} \subset S_n\), \(|\partial \mathcal{A}| \geq |\partial \mathcal{C}|\), where \(\mathcal{C}\) denotes the initial segment of the lexicographic order on \(S_n\) of size \(|\mathcal{A}|\). 
\end{conjecture}

This is a beautiful conjecture; it may be compared to the edge-isoperimetric inequality in \(\{0,1\}^n\), due to Harper \cite{harper}, Lindsey \cite{lindsey}, Bernstein \cite{bernstein} and Hart \cite{hart}, stating that among all subsets of \(\{0,1\}^n\) of size \(k\), the first \(k\) elements of the binary ordering on \(\{0,1\}^n\) has the smallest edge boundary. (Recall that if \(x,y \in \{0,1\}^n\), we say that \(x < y\) {\em in the binary ordering} if \(x_j = 0\) and \(y_j = 1\), where \(j = \min\{i \in [n]\colon x_i \neq y_i\}\).)

In this section, we use eigenvalue techniques to prove an approximate version of Conjecture \ref{conj:benefraim}; this version is asymptotically sharp for sets of size $n!/\textrm{poly}(n)$. We then combine eigenvalue techniques with Theorem \ref{thm:main} to prove the exact conjecture for sets of size \((n-t)!\), for \(n\) sufficiently large depending on \(t\).

For sets of size \(c(n-1)!\), where \(c \in [n]\), Conjecture \ref{conj:benefraim} follows from calculating the second eigenvalue of the Laplacian of \(T_n\). We briefly outline the argument.

If \(G = (V,E)\) is a finite graph and \(A\) is the adjacency matrix of \(G\), the {\em Laplacian matrix} of \(G\) may be defined by 
\begin{equation}\label{eq:lap-def} L = D-A,\end{equation}
where \(D\) is the diagonal \(|V| \times |V|\) matrix with rows and columns indexed by \(V\), and with 
\[D_{u,v} = \left\{\begin{array}{cc} \deg(v) & \textrm{if }u=v\\
0 & \textrm{if } u \neq v.\end{array}\right.\]

The following well-known theorem, due independently to Dodziuk \cite{dodziuk} and Alon and Milman \cite{alonmilman}, provides an edge-isoperimetric inequality for a graph \(G\) in terms of the smallest non-trivial eigenvalue of its Laplacian matrix (the `spectral gap'):

\begin{theorem}[Dodziuk / Alon--Milman]
\label{thm:alon}
Let \(G = (V,E)\) be a finite graph, let \(L\) denote the Laplacian matrix of \(G\), and let \(0=\mu_1 \leq \mu_2 \leq \ldots \leq \mu_{|V|}\) be the eigenvalues of \(L\) (repeated with their multiplicities). If $S \subset V(G)$, then
\[e(S,S^c) \geq \mu_2\frac{|S\|S^c|}{|V|}.\]
If equality holds, then the characteristic vector \(1_{S}\) of \(S\) satisfies
\begin{equation}\label{eq:equality} 1_{S}- \frac{|S|}{|V|}\mathbf{f} \in \ker(L - \mu_2 I),\end{equation}
where \(\mathbf{f}\) denotes the all-1's vector.
\end{theorem}

The proof follows straightforwardly from the formula \[ e(S,S^c) = |V| \langle 1_S, L 1_S \rangle, \]
where $\langle f,g \rangle := \frac{1}{|V|} \sum_{v \in V} f(v) \overline{g(v)}$.

If \(G\) is a finite \(d\)-regular graph (and note that \(T_n\) is \(\binom{n}{2}\)-regular), then (\ref{eq:lap-def}) becomes \(L = dI-A\). Therefore, if the eigenvalues of $A$ are
\[d = \lambda_1 \geq \lambda_2 \geq \ldots \geq \lambda_{|V|},\]
then \(\mu_i = d-\lambda_i\) for each \(i\). In particular, \(\mu_2 = d-\lambda_2\), so for any set \(S \subset V(G)\),
\[e(S,S^c) \geq (d-\lambda_2)\frac{|S\|S^c|}{|V|}.\]

The transposition graph is a normal Cayley graph, and therefore the eigenvalues of its adjacency matrix are given by (\ref{eq:normalcayleygraph}). Frobenius gave the following formula for the value of \(\chi_{\alpha}\) at a transposition, where $\alpha = (\alpha_1,\ldots,\alpha_m)$ is a partition of $n$.
\[\chi_{\alpha}((1 \ 2)) = \frac{\dim(\rho_{\alpha})}{{\binom{n}{2}}} \tfrac{1}{2} \sum_{j=1}^{m}((\alpha_j-j)(\alpha_j-j+1) - j(j-1))\quad (\alpha \vdash n).\]
Combining this with (\ref{eq:normalcayleygraph}) yields the following formula for the eigenvalues of the adjacency matrix of the transposition graph.
\begin{equation}\label{eq:evalstrans} \lambda_{\alpha} = \tfrac{1}{2} \sum_{j=1}^{m}((\alpha_j-j)(\alpha_j-j+1) - j(j-1))\quad (\alpha \vdash n).\end{equation}
Note that \(\lambda_{(n)} = \binom{n}{2}\), \(\lambda_{(n-1,1)} = \binom{n}{2}-n\) and \(\lambda_{(n-2,2)} = \binom{n}{2} - 2n+2\). Diaconis and Shashahani \cite{diaconis} verify that if \(\alpha\) and \(\alpha'\) are two partitions of \(n\), then
\begin{equation}\label{eq:dom} \alpha \domgeq \alpha'\quad \Rightarrow \quad \lambda_{\alpha} \geq \lambda_{\alpha'}.\end{equation}
%Since \((n-2,2) \domgeq \alpha\) whenever \(\alpha_1 \leq n-2\), we have \(\lambda_{\alpha} \leq \lambda_{(n-2,2)} = \binom{n}{2}-2n+2\) whenever \(\alpha_2 \leq n-2\).
Hence, if \(n \geq 2\), then \(\mu_2=n\), and the \(\mu_2\)-eigenspace of the Laplacian is precisely \(U_{(n-1,1)}\). Theorem \ref{thm:alon} therefore yields the following.

\begin{theorem}[essentially due to Diaconis and Shahshahani]
\label{thm:diaconis}
If \(\mathcal{A} \subset S_n\), then
\[|\partial \mathcal{A}| \geq \frac{|\mathcal{A}|(n!-|\mathcal{A}|)}{(n-1)!}.\]
Equality holds only if
\begin{equation}\label{eq:equality2} 1_{\mathcal{A}} \in U_{(n)} \oplus U_{(n-1,1)}.\end{equation}
\end{theorem}

Observe that equality holds in Theorem \ref{thm:diaconis} if and only if \(\mathcal{A}\) is a disjoint union of 1-cosets of \(S_n\). (The `only if' part follows from (\ref{eq:equality2}) and Theorem \ref{thm:char}.)

When \(|\mathcal{A}| = o(n!)\), Theorem \ref{thm:diaconis} merely implies \(|\partial \mathcal{A}| \geq (1-o(1))n|\mathcal{A}|\), whereas Conjecture \ref{conj:benefraim} would imply that \(|\partial \mathcal{A}| \geq (1-o(1))n(t+1)|\mathcal{A}|\) whenever \(|\mathcal{A}| = o((n-t)!)\). Our first aim is to prove the latter when \(t\) is small; specifically, we prove the following.

\begin{theorem}
\label{thm:approxiso}
Let \(\mathcal{A} \subset S_n\) with \(|\mathcal{A}| \leq (n-t+1)!\). Then
\[|\partial \mathcal{A}| \geq (1-\tilde{C}_t|\mathcal{A}|/(n-t+1)!) t(n-t+1) |\mathcal{A}|,\]
where \(\tilde{C}_t >0\) depends upon \(t\) alone.
\end{theorem}
\begin{proof}
Let \(f = 1_{\mathcal{A}}\) and let $L$ be the Laplacian of $T_n$. By Lemma \ref{lemma:smallnorm}, we have
\[\|f_{t-1}\|_2^2 \leq C_t' n^{t-1} (|\mathcal{A}|/n!)^2 = O_{t}(|\mathcal{A}|/(n-t+1)!)|\mathcal{A}|/n!.\]
Therefore, we have
\[\sum_{\substack{\alpha \vdash n\colon\\\ \alpha_1 \leq n-t}} \|f_{\alpha}\|_2^2 = \|f\|_{2}^2 - \|f_{t-1}\|_2^2 \geq (1-O_t(|\mathcal{A}|/(n-t+1)!))|\mathcal{A}|/n!.\]
Note that if \(\alpha_1 \leq n-t\), then \(\alpha \domleq (n-t,t)\), and therefore \(\lambda_{\alpha} \leq \lambda_{(n-t,t)}\), by (\ref{eq:dom}). By (\ref{eq:evalstrans}), we have
\[\lambda_{(n-t,t)} = \binom{n}{2}-t(n-t+1),\]
so \(\lambda_{\alpha} \leq \binom{n}{2} - t(n-t+1)\) whenever \(\alpha_1 \leq n-t\). In other words, \(\mu_{\alpha} \geq t(n-t+1)\) whenever \(\alpha_1 \leq n-t\). We have
\begin{align*} |\partial \mathcal{A}| & = n! \langle f,Lf \rangle \\
& = n! \sum_{\alpha \vdash n} \mu_{\alpha}\|f_{\alpha}\|_2^2 \\
& \geq n! \sum_{\alpha_1 \leq n-t} \mu_{\alpha} \|f_{\alpha}\|_2^2\\
& \geq n! t(n-t+1) (1-O_{t}(|\mathcal{A}|/(n-t+1)!))|\mathcal{A}|/n!\\
& \geq (1-O_t(|\mathcal{A}|/(n-t+1)!)) t(n-t+1) |\mathcal{A}|,\end{align*}
as required.
\end{proof}

Note that when \(t\) is fixed and \(\mathcal{A} \subset S_n\) with \(|\mathcal{A}| = o((n-t+1)!)\), Theorem \ref{thm:approxiso} yields
\[|\partial \mathcal{A}| \geq (1-o(1)) tn |\mathcal{A}|.\]
If
\[\mathcal{A} = \{\sigma \in S_n : \sigma(i)=i \ \forall i \in [t-1],\ \sigma(t) \in \{t,\ldots,t+m-1\}\},\]
then \(|\mathcal{A}| = m(n-t)!\) and
\begin{align*}
|\partial \mathcal{A}| &= (t(n-t)+t(t-1)/2-m+1)|\mathcal{A}| \\ &=
(t(n-(t+1)/2)-m+1)|\mathcal{A}|.
\end{align*}
If \(m = o(n)\) then for constant $t$ we obtain
\[|\partial \mathcal{A}| = (1-o(1))tn|\mathcal{A}|,\]
showing that Theorem \ref{thm:approxiso} is asymptotically sharp in this case.

We now use a similar argument, together with Theorem \ref{thm:main}, to prove Theorem \ref{thm:iso}, thus verifying Conjecture \ref{conj:benefraim} for sets of size \((n-t)!\), when \(n\) is sufficiently large depending on \(t\).

\begin{proof}[Proof of Theorem \ref{thm:iso}.]
Let \(\mathcal{A} \subset S_n\) with \(|\mathcal{A}|=(n-t)!\). Note that
$$|\partial T_{(1,2,\ldots,t)(1,2,\ldots,t)}| = t(n - (t+1)/2)(n-t)!.$$
Assume that \(|\partial \mathcal{A}| \leq |\partial T_{(1,2,\ldots,t)(1,2,\ldots,t)}|\); we will show that \(\mathcal{A}\) must be a \(t\)-coset.

Let \(f = 1_{\mathcal{A}}\); then \(\|f\|_2^2 = 1/(n)_t\). Write
$$\|f - f_t\|_2^2 = \phi \|f\|_2^2 = \phi |\mathcal{A}|/n!,$$
where \(\phi \in [0,1]\). Our first aim is to show that \(\phi\) must be small. 

By Lemma \ref{lemma:smallnorm}, we have
\[\|f_{t-1}\|_2^2 \leq C_t' n^{t-1}(|\mathcal{A}|/n!)^2 = O_{t}(1/n)|\mathcal{A}|/n!.\]
Therefore, we have
\[\sum_{\substack{\alpha \vdash n\colon\\\ \alpha_1=n-t}}\|f_\alpha\|_2^2 = \|f_t\|_2^2 - \|f_{t-1}\|_2^2 = (1-\phi - O_t(1/n))|\mathcal{A}|/n!.\]

Writing $L$ for the Laplacian of $T_n$, we have
\begin{align*} |\partial \mathcal{A}| & = n! \langle f,Lf \rangle\\
& = n! \sum_{\alpha \vdash n}\mu_{\alpha} \|f_{\alpha}\|_2^2\\
& \geq n! \sum_{\alpha_1 = n-t} \mu_{\alpha} \|f_\alpha\|_2^2 + n! \sum_{\alpha_1 \leq n-t-1} \mu_{\alpha} \|f_{\alpha}\|_2^2\\
& \geq n! t(n-t+1) (1-\phi-O_t(1/n))|\mathcal{A}|/n! + n!  (t+1)(n-t)\phi|\mathcal{A}|/n!\\
& = [t(n-t+1)(1-\phi-O_t(1/n))+(t+1)(n-t)\phi](n-t)!\\
& = (tn + \phi n -O_t(1))(n-t)!.
\end{align*}
Since we are assuming that \(|\partial \mathcal{A}| \leq t((n-(t+1)/2)(n-t)!\), it follows that \(\phi = O_t(1/n)\). Hence, we may apply Theorem \ref{thm:main} (provided $n$ is sufficiently large depending on $t$) with \(\epsilon = O_t(1/n)\). We see that there exists a $t$-coset \(\mathcal{C} \subset S_n\) such that
\[|\mathcal{A} \triangle \mathcal{C}| = O_t(1/\sqrt{n}) (n-t)!.\]
Our aim is to show that \(\mathcal{A} = \mathcal{C}\). Let us write
\[|\mathcal{A} \setminus \mathcal{C}| = \psi (n-t)!, \quad \psi = O_t(1/\sqrt{n}). \]
Let \(\mathcal{E} = \mathcal{A} \setminus \mathcal{C}\), and let \(\mathcal{M} = \mathcal{C} \setminus \mathcal{A}\); then
\[|\mathcal{E}| = |\mathcal{M}| = \psi (n-t)!.\]
Let $\mathcal{X} = \mathcal{A} \cap \mathcal{C}$ and $\mathcal{Y} = S_n \setminus (\mathcal{A} \cup \mathcal{C})$, so that $\mathcal{A} = \mathcal{E} \cup \mathcal{X}$ and $\mathcal{C} = \mathcal{M} \cup \mathcal{X}$.

Observe that
\begin{align*}
\boundary{\cA} & = \edges{\cA}{\cY} + \edges{\cA}{\cM}\\
& = [\edges{\cC}{\cY} + \edges{\cE}{\cY} - \edges{\cM}{\cY}] \; + \; [\edges{\cE}{\cM}+\edges{\cX}{\cM}] \\
& = [\edges{\cC}{\cY} + \edges{\cC}{\cE}] + [\edges{\cC}{\cE} + \edges{\cE}{\cY}] - 2\edges{\cC}{\cE} \\
& - [\edges{\cM}{\cY} + \edges{\cM}{\cE}]+2\edges{\cM}{\cE}+\edges{\cX}{\cM}\\
& = \boundary{\cC} + \boundary{\cE} - 2\edges{\cC}{\cE} - \edges{\cM}{S_n\setminus\cC} + 2\edges{\cE}{\cM} + \edges{\cM}{\cX}.
\end{align*}
In particular,
\begin{equation} \label{eq:decomposition}
 \boundary{\cA} \geq
 \boundary{\cC} + \boundary{\cE} - 2\edges{\cC}{\cE} - \edges{\cM}{S_n\setminus\cC} + \edges{\cM}{\cX}.
\end{equation}

% Observe that
% \begin{align}
% \label{eq:decomposition} |\partial \mathcal{A}| &= |\partial \mathcal{C}| + |\partial \mathcal{E}|- 2e(\mathcal{C},\mathcal{E}) - e(\mathcal{M},S_n \setminus (\mathcal{C} \cup \mathcal{E})) + e(\mathcal{M}, \mathcal{C} \setminus \mathcal{M}) + e(\mathcal{E},\mathcal{M}) \nonumber \\
% & = |\partial \mathcal{C}| + |\partial \mathcal{E}| - 2e(\mathcal{C},\mathcal{E})- e(\mathcal{M},S_n \setminus \mathcal{C}) + e(\mathcal{M},\mathcal{C} \setminus \mathcal{M}) + 2e(\mathcal{E}, \mathcal{M})\nonumber \\
% & \geq |\partial \mathcal{C}| + |\partial \mathcal{E}| - 2e(\mathcal{C},\mathcal{E})- e(\mathcal{M},S_n \setminus \mathcal{C}) + e(\mathcal{M},\mathcal{C} \setminus \mathcal{M}).
% \end{align}

We will show that $\boundary{\cE} - 2\edges{\cC}{\cE} - \edges{\cM}{S_n\setminus\cC} + \edges{\cM}{\cX} > 0$ unless $\psi = 0$. Let us bound each of these terms in turn. First, we bound \(|\partial \mathcal{E}|\). Since \(|\mathcal{E}| \leq (n-t)!\), Theorem \ref{thm:approxiso} yields 
\[|\partial \mathcal{E}| \geq (1-O_t(1/n)) t(n-t+1) |\mathcal{E}| =(1-O_t(1/n)) tn \psi (n-t)!.\]

Next, we bound \(e(\mathcal{C},\mathcal{E})\). By definition, we have \(\mathcal{E} \cap \mathcal{C} = \emptyset\). Without loss of generality, we may assume that \(\mathcal{C} = \{\sigma \in S_n : \sigma(i)=i\ \forall i \in [t]\}\). For any \(\sigma \in S_n \setminus \mathcal{C}\), choose \(i \in [t]\) such that \(\sigma(i) \neq i\). If \(\sigma (i\ \sigma^{-1}(i)) \in \mathcal{C}\), then it is the unique neighbour of \(\sigma\) in \(\mathcal{C}\); otherwise, \(\sigma\) has no neighbour in \(\mathcal{C}\). It follows that
\[e(\mathcal{E},\mathcal{C}) \leq |\mathcal{E}| = \psi (n-t)!.\]

Next, we calculate \(e(\mathcal{M},S_n \setminus \mathcal{C})\). Observe that each \(\sigma \in \mathcal{C}\) has exactly \(t(n-t) + t(t-1)/2\) neighbours in \(S_n \setminus \mathcal{C}\). Indeed, the neighbours of \(\sigma\) in \(S_n \setminus \mathcal{C}\) are precisely \(\{\sigma(i\ j) : i \in [t], j \notin [t] \} \cup \{\sigma(i\ j) : i,j \in [t], j \neq i\}\). It follows that
\[e(\mathcal{M},S_n \setminus \mathcal{C}) = t(n-(t+1)/2) |\mathcal{M}| \leq tn |\mathcal{M}|.\]

Finally, we bound \(e(\mathcal{M}, \mathcal{X})\). Since \(T_n[\mathcal{C}]\) is isomorphic to \(T_{n-t}\), applying Theorem \ref{thm:diaconis} to \(T_{n-t}\) yields:
\[e(\mathcal{C} \setminus \mathcal{A},\mathcal{A} \cap \mathcal{C}) \geq \psi (n-t)!(1-\psi)(n-t)!/(n-t-1)! = \psi (1-\psi) (n-t)(n-t)!.\]

Subsituting all of these bounds into (\ref{eq:decomposition}), we obtain:
\begin{align*} |\partial \mathcal{A}| & \geq |\partial \mathcal{C}| + (1-O_t(1/n)) tn \psi (n-t)!\\
& - 2\psi (n-t)! - tn\psi(n-t)!+\psi (1-\psi) (n-t)(n-t)!\\
& =|\partial \mathcal{C}| + \big[(1-O_t(1/n))tn -2 -tn +(1-\psi)(n-t)\big]\psi(n-t)!\\
& = |\partial \mathcal{C}| + [(1-O_t(1/\sqrt{n}))n-O_t(1)]\psi(n-t)!.
\end{align*}

If \(\psi >0\) and \(n\) is sufficiently large depending on \(t\), then the right-hand side is greater than \(|\partial \mathcal{C}|\), contradicting our assumption that \(|\partial \mathcal{A}| \leq |\partial \mathcal{C}|\). It follows that \(\psi = 0\) and therefore \(\mathcal{A} = \mathcal{C}\), proving Theorem \ref{thm:iso}.
\end{proof}

\section{A note on \(t\)-intersecting families of permutations}
In \cite{EFP}, we showed how a \(t=1\) version of Theorem \ref{thm:main} could be used to give a natural proof of the Cameron--Ku conjecture on large intersecting families of permutations. Let us briefly give the background to this result. Recall that a family of permutations \(\mathcal{A} \subset S_n\) is said to be {\em intersecting} if for any two permutations \(\sigma,\pi \in \mathcal{A}\), there exists \(i \in [n]\) such that \(\sigma(i)=\pi(i)\). Using a simple partioning argument, Deza and Frankl proved that if \(\mathcal{A} \subset S_n\) is intersecting, then \(|\mathcal{A}| \leq (n-1)!\), i.e. the 1-cosets are intersecting families of maximum size. Cameron and Ku \cite{cameron} proved that equality holds only if \(\mathcal{A}\) is a 1-coset of \(S_n\). They made the following `stability' conjecture.
\begin{conjecture}[Cameron--Ku, 2003]
\label{conj:cameronku}
There exists \(c>0\) such that if \(\mathcal{A} \subset S_n\) is an intersecting family with \(|\mathcal{A}| \geq c(n-1)!\), then \(\mathcal{A}\) is contained within a 1-coset.
\end{conjecture}
Conjecture \ref{conj:cameronku} was first proved by the first author in \cite{cameronkuconj}, but the proof in \cite{EFF1} is, in a sense, more natural.

We say that a family of permutations \(\mathcal{A} \subset S_n\) is \(t\)-{\em intersecting} if for any \(\sigma,\pi \in \mathcal{A}\), there exist \(i_1,\ldots,i_t \in [n]\) such that \(\sigma(i_k)=\pi(i_k)\) for all \(k \in [t]\). In 1977, Deza and Frankl \cite{dezafrankl} conjectured that if \(n\) is sufficiently large depending on \(t\), then any \(t\)-intersecting family in \(S_n\) has size at most $(n-t)!$. This is proved in \cite{EFP}. In \cite{tstability}, the first author proved the following analogue of Conjecture \ref{conj:cameronku} for \(t\)-intersecting families of permutations.
\begin{theorem}
\label{thm:tcameronku}
For any \(t \in \mathbb{N}\), there exists $c<1$ and $n_0 \in \mathbb{N}$ such that if \(n \geq n_0\), then any \(t\)-intersecting family \(\mathcal{A} \subset S_n\) with \(|\mathcal{A}| \geq c(n-t)!\) is contained within a \(t\)-coset.
\end{theorem}
 
Almost exactly as in the \(t=1\) case, Theorem \ref{thm:main} can be used to give a more `natural' proof of Theorem \ref{thm:tcameronku}. We do not give the details here; it suffices to say that one simply replaces the adjacency matrix of the derangement graph (used in the \(t=1\) case) with the `weighted' analogue constructed in \cite{EFP}.

We note that Theorem \ref{thm:tcameronku} implies that if $n$ is sufficiently large depending on $t$, then any \(t\)-intersecting family \(\mathcal{A} \subset S_n\) with \(|\mathcal{A}| = (n-t)!\) must be a \(t\)-coset, i.e.\ the equality case of the Deza-Frankl conjecture. In \cite{EFP}, this is deduced from \cite[Theorem 27]{EFP}, but as stated on page \pageref{exp} (in the Introduction) of the current paper, that theorem is false. As mentioned in the Introduction, one may deduce the equality case of the Deza-Frankl conjecture from the Hilton-Milner type result in \cite{tstability}. Alternatively, of course, one may also deduce it from Theorem \ref{thm:tcameronku}, using Theorem \ref{thm:main} as a `black box' to prove Theorem \ref{thm:tcameronku}. While this alternative proof is arguably more natural than the proof in \cite{tstability}, it is not `truly' shorter, as the proof of Theorem \ref{thm:main} in the current paper is longer than the proof of the Hilton-Milner type result in \cite{tstability}.

\section{Conclusion and open problems}
We conjecture the following strengthening of Theorem \ref{thm:main}.
\begin{conjecture}
\label{conj:main}
Let $\mathcal{A} \subset S_n$, and let $t \in \mathbb{N}$. Let $f$ denote the characteristic function of $\mathcal{A}$, and let $f_t$ denote the orthogonal projection of $f$ onto $U_t$. If
$$\mathbb{E}[(f-f_t)^2] \leq \epsilon \mathbb{E}[f],$$
then there exists $\mathcal{C} \subset S_n$ such that $\mathcal{C}$ is a union of $t$-cosets, and
$$|\mathcal{A} \triangle \mathcal{C}|  \leq C_0(\epsilon+1/n)|\mathcal{A}|,$$
where $C_0$ is an absolute constant.
\end{conjecture}
We remark that the $1/n$ term would be best possible up to an absolute constant factor. For $t=2$, this can be seen by considering the family
$$\mathcal{A} = \{\sigma \in S_n:\ |\sigma([3]) \cap [3]| = 0 \text{ or }3\}.$$
It is easy to check that $f:=1_{\mathcal{A}} \in U_2$, that
$$\mathbb{E}[f] = 1-\frac{9(n-3)}{n(n-1)} = 1-O(1/n),$$
and that $|\mathcal{A} \Delta \mathcal{C}| = \Omega(1/n)|\mathcal{A}|$ for any family $\mathcal{C} \subset S_n$ such that $\mathcal{C}$ is a union of 2-cosets of $S_n$.

The $\epsilon$-term would also be best possible up to an absolute constant factor, as can be seen by considering a family of the form
$$\mathcal{A} = (T_{(1,2,\ldots,t)(1,2,\ldots,t)} \cup \mathcal{E}) \setminus \mathcal{F}$$
where $\epsilon \leq 1/2$, $\epsilon(n-t)!/2 \in \mathbb{N}$, $\mathcal{F} \subset T_{(1,2,\ldots,t)(1,2,\ldots,t)}$ and $\mathcal{E} \subset S_n \setminus T_{(1,2,\ldots,t)(1,2,\ldots,t)}$ with $|\mathcal{E}| = |\mathcal{F}| = \epsilon(n-t)!/2$. (It is easy to check that $f = 1_{\mathcal{A}}$ satisfies $\mathbb{E}[(f-f_t)^2] \leq \epsilon \mathbb{E}[f]$, and that $|\mathcal{A} \triangle \mathcal{C}|  \geq \epsilon |\mathcal{A}|$ for any $\mathcal{C} \subset S_n$ such that $\mathcal{C}$ is a union of $t$-cosets.)

In section \ref{section:isoperimetric}, we gave an example of how Fourier-analytic arguments can yield sharp isoperimetric inequalities for relatively small sets, even when the classical `eigenvalue gap' inequality in Theorem \ref{thm:alon} is far from sharp. We speculate that this technique of combining eigenvalue arguments, Fourier-analytic arguments and stability arguments may be useful in obtaining other isoperimetric inequalities. Conjecture \ref{conj:benefraim} remains very much open, however; we suspect that more combinatorial techniques are required to prove it in full generality.

It is instructive to contrast the situation for Boolean functions on $\{0,1\}^n$, with the situation for the symmetric group. For Boolean functions on $\{0,1\}^n$, the analogues of the $t$-cosets are the subcubes of codimension $t$, i.e.\ sets of the form $\{x \in \{0,1\}^n:\ x_i = c_i\ \forall i \in S\}$, where $S \in [n]^{(t)}$ and $c_i \in \{0,1\}$ for each $i \in S$. If $f:\{0,1\}^n \to \{0,1\}$, we write $\hat{f}$ for the Fourier transform of $f$, i.e.
$$\hat{f}: \mathcal{P}([n]) \to \mathbb{R};\quad \hat{f}(S) =\frac{1}{2^n} \sum_{x \in \{0,1\}^n} (-1)^{\sum_{i \in S} x_i}\quad \forall S \subset [n].$$
We equip $\mathbb{R}[\{0,1\}^n]$ with the inner product $\langle f,g \rangle = \frac{1}{2^n}\sum_{x \in \{0,1\}^n}f(x)g(x)$, and we let $\| \cdot \|_2$ denote the corresponding Euclidean norm. It is easy to see that if $f:\{0,1\}^n \to \{0,1\}$, then the square of the Euclidean distance from $f$ to the subspace of $\mathbb{R}[\{0,1\}^n]$ spanned by the subcubes of codimension $t$ is precisely $\sum_{S \subset [n]:\ |S|>t} \hat{f}(S)^2$, i.e. the $\ell^2$-mass of the Fourier transform of $f$ on sets of size greater than $t$.

It was proved in \cite{fkn} that if $f:\{0,1\}^n \to \{0,1\}$ has Fourier transform with at most $\epsilon$ of its $\ell^2$-mass on sets of size greater than $t$ (i.e.,\ $f$ has Euclidean distance at most $\sqrt{\epsilon}$ from the subspace of $\mathbb{R}[\{0,1\}^n]$ spanned by the subcubes of codimension 1), then $f$ is $O(\epsilon)$-close to the characteristic function of a subcube of codimension 1, or to a constant function. By contrast, for each even $t \geq 2$ and each $n \geq t+1$, the Boolean function
\begin{equation} \label{eq:counter} f:\{0,1\}^n \to \{0,1\};\quad x \mapsto 1\{|x \cap [t+1] = 0 \text{ or } t+1\}\end{equation}
has Fourier transform supported on sets of size at most $t$, but has $\|f-g\|_2^2 \geq 2^{-t}$ for any $g:\{0,1\}^n \to \{0,1\}$ such that $g$ is the characteristic function of a union of subcubes of codimension $t$. In addition, the function $f$ in (\ref{eq:counter}) has the same expectation as the characteristic function of a subcube of codimension $t$. Theorem \ref{thm:main} rules out the existence of an analogous function on the symmetric group, for large $n$. 

We can, however, say the following about Boolean functions on $\{0,1\}^n$. Kindler and Safra proved in \cite{KindlerSafra} that if $f:\{0,1\}^n \to \{0,1\}$ has Fourier transform with at most $\epsilon$ of its $\ell^2$-mass on sets of size greater than $t$, then $f$ must be $O_t(\epsilon)$-close to a Boolean function $g:\{0,1\}^n \to \{0,1\}$ depending upon at most $j(t)$ coordinates, for some function $j:\mathbb{N} \to \mathbb{N}$. Combined with \cite[Corollary 6.2]{tal}, and an easy compactness argument, this yields the following.
\begin{corollary}
\label{corr:ns}
There exists $s: \mathbb{N} \to \mathbb{N}$ such that the following holds. If $f:\{0,1\}^n \to \{0,1\}$ is such that $\sum_{S \subset [n]: |S|>t} \hat{f}(S)^2 \leq \epsilon$, then there exists a function $h:\{0,1\}^n \to \{0,1\}$ such that $\|f-h\|_2^2 = O_t(\epsilon)$ and $h$ is the characteristic function of a union of subcubes of codimension $s(t)$. Moreover, we may take $s(t) \leq t^3$.
\end{corollary}

The function $f$ in (\ref{eq:counter}) shows that one cannot take $s(t)=t$ in Corollary \ref{corr:ns}, even if $f$ has the same expectation as a subcube of codimension $t$. Theorem \ref{thm:main} provides an $S_n$-analogue of Corollary \ref{corr:ns} with $s(t)=t$, albeit with $O_t(\sqrt{\epsilon} + c/\sqrt{n})c/(n)_t$ in place of $O_t(\epsilon)$. Conjecture \ref{conj:main} would provide such an analogue with $s(t)=t$, albeit with $O(\epsilon+1/n)c/(n)_t$ in place of $O_t(\epsilon)$.

\subsection*{Acknowledgements}
The authors wish to thank Gil Kalai for many useful discussions, and an anonymous referee for their careful reading of the paper and their helpful comments.

\end{document}